\RequirePackage{fix-cm}
\documentclass[smallextended]{svjour3}       
\smartqed  
\usepackage{graphicx,amsmath}
\usepackage{lipsum}
\usepackage{enumerate,algorithmicx,algorithm}
\usepackage{threeparttable}
\usepackage{algpseudocode}
\usepackage{bm}
\usepackage{amssymb}
\usepackage{booktabs}
\usepackage{multirow} 
\usepackage{verbatim}

\begin{document}

\title{Reduced projection method for photonic moir\'e lattices
}


\author{Zixuan Gao      \and
        Zhenli Xu \and Zhiguo Yang 
}


\institute{Zixuan Gao     \and
          Zhenli Xu      \and
              Zhiguo Yang\at
              School of Mathematical Sciences, MOE-LSC and CMA-Shanghai, Shanghai Jiao Tong University, Shanghai 200240, China. \\
              \email{1270157606gzx@sjtu.edu.cn} 
}

\date{Received: date / Accepted: date}

\maketitle

\begin{abstract}
This paper presents a reduced projection method for the solution of quasiperiodic Schr\"{o}dinger eigenvalue problems for photonic moir\'e lattices. Using the properties of the Schr\"{o}dinger operator in higher-dimensional space via a projection matrix, we rigorously prove that the generalized Fourier coefficients of the eigenfunctions exhibit faster decay rate along a fixed direction associated with the projection matrix.  An efficient reduction strategy of the basis space is then proposed to reduce the degrees of freedom significantly. Rigorous error estimates of the proposed reduced projection method are provided, indicating that a small portion of the degrees of freedom is sufficient to achieve the same level of accuracy as the classical projection method. We present numerical examples of photonic moir\'e lattices in one and two dimensions to demonstrate the accuracy and efficiency of our proposed method. 
\keywords{Quasiperiodic problems \and Schr\"odinger eigenvalue problems \and reduced projection method \and Fourier method \and basis reduction.}
\subclass{65N35 \and 65N22 \and 65F05 \and 35J05}
\end{abstract}

\section{Introduction}
The quasiperiodic problems emerge naturally in a great many physical systems such as quasicrystals, many-body problems, and low-dimensional materials \cite{wang2020localization,cao2018unconventional,gao2023pythagoras,lu2019superconductors,zhang2021quasi}, and have found numerous applications in the areas of mechanics, acoustics, electronics, solid-state physics, and physics of matter waves \cite{bistritzer2011moire,o2016moire,hu2020moire}. Efficient and accurate numerical simulations of quasiperiodic problems play a critical role in  exploring and utilizing novel material properties. 

Though quasiperiodic systems are ubiquitous in mathematics and physics, the numerical method is not as straightforward as that of periodic systems. Specifically, quasiperiodic systems are space-filling ordered without decay and translational invariance \cite{jiang2024numerical}. {In recent years, there has been a growing interest in the research of the optical properties of moir\'e lattices, a prototype of quasicrystals, as evidenced by notable studies in \cite{fu2020optical,kartashov2021multifrequency,salakhova2021fourier} that have brought about significant breakthroughs in the field of optics. It is exhilarating to note that a localization-to-delocalization transition of eigenstates of moir\'{e} lattices in two dimensions is observed for the first time in both numerical simulations and experiments \cite{wang2020localization},  which paves a new way of controlling light at will. However, the localization of eigenstates as well as the phase transition for a high-dimensional case is not well explored, due to the exceedingly huge degrees of freedom and computational cost required.}

One of the widely used numerical approaches for solving these quasiperiodic problem is the periodic approximation method, also known as the crystalline approximant method \cite{davenport1946simultaneous,goldman1993quasicrystals,lifshitz1997theoretical}, which approximates the quasiperiodic function via a periodic function in a certain supercell.  Nevertheless, this method is proven to be of slow convergence and the simultaneous Diophantine approximation error does not  decay uniformly as the size of the supercell gradually increases.  In order to avoid the Diophantine error,  considerable efforts have been made. Rodriguez {\it et al.} \cite{rodriguez2008computation} introduced a numerical algorithm to compute the spectrum of photonic quasicrystals by raising the physical domain to higher dimensions. Jiang and Zhang \cite{jiang2014numerical,jiang2018numerical} proposed a projection method (PM), which treats a quasiperiodic function as a projection of a higher-dimensional periodic function. The PM has the advantage of  avoiding the simultaneous Diophantine approximation error and allowing for the convenient use of periodic boundary conditions for domain truncation.  Despite the high accuracy, the PM requires solving problems in higher dimensions, leading to significant increases in the degrees of freedom (DOF), computational cost, and memory consumption. For instance, when solving a $d$-dimensional quasiperiodic eigenvalue problem, the projection method raises the $d$-dimensional quasiperiodic domain to an $n$-dimensional periodic domain ($n$ is often twice as large as $d$). Correspondingly, the DOF of the PM is $O(N^n)$, where $N$ is the number of Fourier grid points in one direction. This makes the PM prohibitive for solving high-dimensional quasiperiodic problems. 

In order to mitigate the curse of dimensionality and improve significantly the capability for solving the problem of high-dimensional photonic moir\'e lattices, we propose  an efficient reduced projection method (RPM). This method is inspired by several pioneering works based on the PM for incommensurate problems. Wang {\it et al.} \cite{wang2022convergence} characterizes the density of states of Schr\"odinger operators in the weak sense for the incommensurate system and proposes numerical methods based on the planewave discretization and reciprocal space sampling. Zhou {\it et al.} \cite{zhou2019plane} propose a $k$-points sampling reduction technique under the planewave framework for the electronic structure-related eigenvalue problems of the incommensurate systems. Jiang {\it et al.} \cite{jiang2023accurately} propose a finite points recovery (FPR) method for non-smooth quasiperiodic systems. We highlight the novelties and main contributions of our work as follows:
\begin{itemize}
    \item We propose a fast algorithm for the PM for solving  quasiperiodic Schr\"{o}dinger eigenvalue problems describing photonic moir\'e lattices, which significantly reduces the DOF for the eigenvalue computation. The computational complexity of the proposed RPM for solving the first $k$ eigenpairs using the Krylov subspace method is reduced from $O(kN^{2n})$ to $O(kN^{2(n-d)}D^{2d})$.  
    \item We study the theoretical reliability of the RPM. We prove that the generalized Fourier coefficients of the eigenfunctions exhibit faster decay rate along a fixed direction associated with the projection matrix. Furthermore, a rigorous convergence analysis for the RPM is then presented to demonstrate the high efficiency of the proposed reduction approach.
    \item The RPM is applied to solve one-dimensional and two-dimensional photonic  moir\'e lattice  systems \cite{wang2020localization}. Numerical experiments showcase the effectiveness of the proposed RPM and demonstrate that it is promising for quasi periodic problems. 
\end{itemize} 

The rest of this paper is organized as follows. Section \ref{s2} presents some preliminary results on quasiperiodic functions. Section \ref{s3} proposes the RPM. By introducing the variational framework, the numerical method for quasiperiodic Schr\"odinger eigenvalue problems is described in details together with the error estimation. Section \ref{s4} presents numerical results to show the attractive performance of the algorithm. Section \ref{s5} makes conclusions with some closing remarks. 

\section{Preliminaries on quasiperiodic functions}\label{s2}

We denote $\mathbb R$, $\mathbb Q$, $\mathbb Z$,  $\mathbb{Z}^{+}$ as the spaces of real, rational, integer, and positive integer numbers, respectively. Let $L^2(\Omega)$ and $H^m(\Omega)$ ($m\in\mathbb{Z}, m\geq0$) be the usual square-integrable function space and Sobolev spaces following classic textbooks (see e.g. \cite{adams2003sobolev}).  Given that there would be a frequent conversion between periodic and quasiperiodic functions, we adopt the subscripts ``per'' and ``qp'', respectively, to differentiate the spaces they belong to and their associated inner products and norms. For instance, for a $d$-dimensional periodic function $F(\bm{z})\in L^2_{\rm per}([0,T]^d)$ with period $T$ in each dimension (dubbed as ``$T$-periodic function''), its corresponding inner product and norm are denoted by 
\begin{equation*}
    (F,G)_{\rm per}=\dfrac{1}{T^d}\int_{[0,T]^d}F(\bm z)\Bar{G}(\bm z){d}\bm{z},\quad \| F\|_{\rm per}=\sqrt{(F,F)_{\rm per}},
\end{equation*}
where $\Bar{G}$ is the complex conjugate of $G\in L^2_{\rm per}([0,T]^d)$. For notational convenience, we omit the subscript ``per'' for periodic functions, if no ambiguity occurs. 

To facilitate the development of efficient and accurate numerical methods for the quasiperiodic Schr\"{o}dinger eigenvalue problem, we begin with a brief exposition of definition for quasiperiodic functions and their relevant properties (see monographs \cite{bohr2018almost,levitan1982almost} for comprehensive discussions).
\begin{definition}\label{df1}
    A $d$-dimensional function $f(\bm{z})$ is quasiperiodic if there exists a $d\times n$ projection matrix $\mathbf{P}$ such that $F(\bm{x})=F(\mathbf{P}^\intercal\bm{z})=f(\bm{z})$ is an $n$-dimensional periodic function, where all columns of $\mathbf P$ are linearly independent over $\mathbb{Q}$. $F(\bm x)$ is called the parent function of $f(\bm z)$ with respect to $\mathbf P$. 
\end{definition}
It is worthwhile to point out that the projection matrix $\mathbf P$ is not unique. Throughout the paper, $\mathbf P$ is chosen such that $F(\bm x)$ is $2\pi$-periodic. Define the  mean value of a $d$-dimensional quasiperiodic function $f(\bm{z})$ by
    \begin{equation}
        \mathcal{M}(f)=\lim_{L\rightarrow\infty}\dfrac{1}{|L|^d}\int_{K}f(\bm{z}){d}\bm{z},
    \end{equation}
    where $K=\{\bm{z}\,| \,0\leq |\bm{z}_i|\leq L,i=1,\dots,d\}$. Correspondingly, one can define the square-integral quasiperiodic function space $L^2_{\rm qp}(\mathbb{R}^n)$ as
    \begin{equation}
L^2_{\rm qp}(\mathbb{R}^n):=\{f(\bm z) \,|\;\mathcal{M}(f\bar{f})<\infty   \},
    \end{equation}
with the inner product and norm defined by
\begin{equation}
(f,g)_{\rm qp}=\mathcal{M}(f\bar{g}),\quad \|f \|_{\rm qp}=\sqrt{  (f,f)_{\rm qp}   }\,,
\end{equation}    
where $f$ and $g$ are quasiperiodic functions with respect to the same projection matrix $\mathbf P$. 

It is well-known that $\{ e^{\mathrm{i}\langle\bm{k},\bm{x}\rangle} \}_{ \bm{k}\in\mathbb{Z}^n}$ serves as a complete orthonormal basis for  $L^2_{\rm per}([0,T]^n)$ such that for any $F(\bm{x})\in L^2_{\rm per}([0,T]^n)$, its has the Fourier series
\begin{equation}\label{dfeq1}
F(\bm{x})=\sum_{\bm{k}/2\pi\in\mathbb{Z}^n}F_{\bm{k}}e^{\mathrm{i}\langle\bm{k},\bm{x}\rangle},\quad     F_{\bm{k}}=\dfrac{1}{T^n}\int_{[0,T]^n}F(\bm{x})e^{-\mathrm{i}\langle\bm{k},\bm{x}\rangle}d\bm{x},
\end{equation}
and there holds the Parseval's equality $ \|F \|^2=\sum_{\bm k/2\pi\in \mathbb{Z}^n} |F_{\bm{k}}|^2$. Lemma \ref{lm: decay} relates the decay rate of Fourier coefficients with the regularity of a function (see \cite[p.196]{grafakos2008classical}).
\begin{lemma}\label{lm: decay}
Let $m\in\mathbb Z^+$, suppose $F(\bm{x})\in H^m_{\rm per}(\mathbb [0,T]^n)$, then
\begin{equation}\label{decay}
    |F_{\bm k}|\leq (\sqrt{n})^{{m}}T^{-n}|\bm k|^{-m}|F|_{m,\rm per},\quad |F|_{m,\rm per}^2=\sum_{\bm k/2\pi\in\mathbb Z^d}\|\bm k\|_2^{2m}|F_{\bm k}|^2.
\end{equation}
\end{lemma}

Quasiperiodic function also has the generalized Fourier series and Parseval's equality. One readily verifies $\big\{e^{\mathrm{i} \langle \bm{q},\bm{z} \rangle} \big \}_{\bm{q}\in \mathbb{R}^d}$ forms a normalized orthogonal system as
\begin{equation}
\big(e^{\mathrm{i}  \langle \bm{q}_1,\bm{z} \rangle        },e^{\mathrm{i}\langle \bm{q}_2,\bm{z} \rangle} \big)_{\rm qp}=\delta_{\bm{q}_1,\bm{q}_2},\quad \bm{q}_1,\bm{q}_2\in \mathbb{R}^d,
\end{equation}
where $ \delta_{\bm{q}_1,\bm{q}_2}$ is the Dirac delta function.  Thus, one can define the Fourier transform of the quasiperiodic functions, also called the Fourier-Bohr transformation (see \cite{bohr2018almost}), as 
\begin{equation}
    \mathcal{F}_{\rm qp}\{f\}(\bm{q})=\mathcal{M}\big(f(\bm z)e^{-\mathrm{i}\langle\bm q,\bm z\rangle}\big),
\end{equation}
Correspondingly, one has the generalized Fourier series of the quasiperiodic function and the Parseval's equality in Lemma \ref{thm2}. 
\begin{lemma}\label{thm2}{(see \cite{bohr2018almost})}
    Any $d$-dimensional quasiperiodic function $f(\bm{z})$ has generalized Fourier series
    \begin{equation}\label{gfs}
        f(\bm{z})=\sum_{\bm q\in\mathbb Z[{\rm col}(\mathbf{P})]}f_{\bm q}e^{\mathrm{i}  \langle \bm{q},\bm{z} \rangle},\quad \mathbb Z[{\rm col}(\mathbf{P})]:=\{\bm q \, | \,\bm q=\mathbf{P}\bm k,\;\; \bm k\in  \mathbb{Z}^n     \},
    \end{equation}
    where $f_{\bm q}=(u,e^{\mathrm{i}  \langle \bm{q},\bm{z} \rangle})_{\rm qp}$ is called the $\bm q$th generalized Fourier coefficient of $f(\bm z)$.
     In addition, there holds the Parseval's equality:
      \begin{equation}
        \|f\|^2_{\rm qp}=\sum_{\bm q\in\mathbb Z[{\rm col}(\mathbf{P})]}|f_{\bm q}|^2.
    \end{equation}
\end{lemma}

Correspondingly, one defines the Sobolev spaces $H^m_{\rm qp}(\mathbb R^d),m\in\mathbb Z^+$ for quasiperiodic functions
\begin{equation}
    H^m_{\rm qp}(L^2_{\rm per}(\mathbb R^d))=\{f(\bm z)\in \mathbb R^d,\|f(\bm z)\|_{m,\rm qp}<\infty\},
\end{equation}
and its associated norm and semi-norm, 
\begin{equation}
    \|f(\bm z)\|_{m,\rm qp}^2=\sum_{\bm q\in\mathbb Z[{\rm col}(\mathbf{P})]}(1+\|\bm q\|_2^{2m})|f_{\bm q}|^2, \quad |f(\bm z)|_{m,\rm qp}^2=\sum_{\bm q\in\mathbb Z[{\rm col}(\mathbf{P})]}\|\bm q\|_2^{2m}|f_{\bm q}|^2.
\end{equation}

Theorem \ref{cosis} relates the generalized coefficients of quasiperiodic functions to the Fourier coefficients of its parent function (\cite{jiang2022numerical}). 
\begin{theorem}\label{cosis}
    Let $f(\bm x)$ be a $d$-dimensional quasiperiodic function. There exists a parent function $F(\bm x)$, which has a one-to-one correspondence between their Fourier coefficients, i.e.
    \begin{equation}\label{paf}
   f(z)=\sum_{\bm q\in\mathbb Z[{\rm col}(\mathbf{P})]}f_{\bm q}e^{\mathrm{i}  \langle \bm{q},\bm{z} \rangle},\qquad     F(\bm x)=\sum_{\bm k \in\mathbb Z^n}F_{\bm k}e^{{\rm i}\langle\bm k,\bm x\rangle}.
    \end{equation}
    One has
    \begin{equation}
    \quad F_{\bm k}=f_{\bm q}, \;\; {\rm iff} \;\; \bm q=\mathbf{P}\bm k.
    \end{equation}
\end{theorem}

By Theorem \ref{cosis},  one can denote the injective mapping $\psi: \mathbb R^d\rightarrow  \mathbb R^n/[0,2\pi]^n$ such that  $\psi(\bm z)=\mathbf{P}^\intercal \bm z=\bm x$, and one can transform the quasiperiodic function into its periodic parent function. With a slight abuse of notation, we also regard $\psi$ as an operator which maps the differential operator $D$ in $\mathbb{R}^d$ to the corresponding differential operator in $\mathbb{R}^n$, using the rules of partial differentiation. On the other hand, through the bijective mapping $\phi:\mathbb Z^{n}\rightarrow\mathbb Z[{\rm col}(\mathbf{P})]$ s.t. $\phi(\bm k)=\mathbf{P}\bm k=\bm q$, one can obtain a one-to-one correspondence between the generalized Fourier coefficient of the quasiperiodic function and Fourier series of the periodic function. 

\section{Reduced projection method}\label{s3}
In this section, we shall propose  an efficient RPM for solving the problem of photonic moir\'e lattices, which is described by the following quasiperiodic Schr\"{o}dinger eigenvalue problem \cite{meng2023atomic}:
\begin{equation}\label{eq1}
\mathcal{L}[u]:=-\dfrac{1}{2}\Delta u(\bm z)+v(\bm z)u(\bm z)=    Eu(\bm z),\quad \bm z\in \mathbb{R}^d.
\end{equation}
Here, $\bm z=(z_1,\cdots,z_d)^{\intercal}$ is the physical coordinates in $d$ dimensions, $v(\bm z)$ is a quasiperiodic potential function, and $u(\bm z)$ and  $E$ are respectively the eigenfunction and eigenvalue of the linear Schr\"odinger operator $\mathcal{L}$.  The PM proposed by Jiang and Zhang in \cite{jiang2014numerical} serves as a viable way to solve quasiperiodic eigenvalue problems. The PM transforms  the $d$-dimensional quasiperiodic problem \eqref{eq1} into its corresponding periodic one in $n$-dimensional space through the variable substitution $\psi(\bm z)=\mathbf{P}^\intercal \bm z=\bm x$. To be specific, it suffices to solve the periodic problem
\begin{equation}\label{eq31}
-\frac{1}{2}\varphi(\Delta)U(\bm x)+V(\bm x) U(\bm x)=E U(\bm x),
\end{equation}
where $\varphi(\Delta)$ is given by
\begin{equation}
\varphi(\Delta)=\sum_{i=1}^d\sum_{j,l=1}^{n} P_{ij}P_{il} \dfrac{\partial^2 }{\partial x_j\partial x_l}.
\end{equation}
It is worthwhile to point out that due to the fact that  operator $\varphi(\Delta)$ in Eq. \eqref{eq1} lacks ellipticity, the quasiperiodic Schr\"odinger operator has only continuous spectrum rather than discrete eigenvalues, as pointed out in \cite{simon1982almost}. Thus, when the resolution of the PM increases, the distribution of numerical eigenvalues gradually converges towards the density of states \cite{wang2022convergence}.  Consequently, a numerical eigenvalue of the PM can be viewed as approximation of a specific point within the spectrum.

The PM is a powerful and accurate numerical method for solving quasiperiodic Schr\"{o}dinger eigenvalue problems, however, it suffers from the ``curse of dimensionality''. As the dimension is raised, the DOF required may become extremely large, making it computationally prohibitive and memory-intensive to solve the quasiperiodic eigenvalue problem. For instance, a three-dimensional quasiperiodic
problems with projection matrix of size $3\times 6$, the DOF to deal with is $O(N^6)$, where $N$ is the number of Fourier grid points in one direction. This poses significant challenges in solving high-dimensional eigenvalue problems.

\subsection{Decay rate of the generalized Fourier coefficients} To reduce the computational cost caused by dimension lifting, we propose a highly efficient RPM to tackle this issue. Before introducing the algorithm, we first focus on a decay property of Fourier coefficients which is presented in Theorem \ref{thm6}, as it can guide us in further improving the PM.

\begin{theorem}\label{thm6}
 Let $u(\bm z)$ be the eigenfunction of the $d$-dimensional quasiperiodic Schr\"{o}dinger eigenvalue problem \eqref{eq1} corresponding to $E$, with $u_{\bm q}$ being its $\bm q$th generalized Fourier coefficient, and $v(\bm z)$ and $V(\bm x)$ being the quasiperiodic potential function and its parent function, respectively. Assume $V(\bm x)\in H_{\rm per}^m([0,2\pi]^n)(m\in\mathbb{Z}, \,m\geq 0)$, then for any $\alpha\leq {\rm max}\big\{ m-n+2, 2\big\}$ and $\|\bm q\|_2>4E$, there exists a constant $C_{\alpha}$ such that
\begin{equation}
    |u_{\bm{q}}|\leq C_{\alpha}\|\bm q\|_2^{-\alpha},
\end{equation}
where $C_{\alpha}$ depends on $m,n,\|\mathbf{P} \|,|V|_{{\rm per},m}$ and $\|v\|$. 
\end{theorem}
\begin{proof}
    Without loss of generality, let us set $\|u\|_{\rm qp}=1$. Define
    \begin{equation}
        g(\bm{z})=v(\bm z)u(\bm z)=\sum_{\bm{q}\in\mathbb Z[{\rm col}(\mathbf{P})]}g_{\bm{q}}e^{\mathrm{i}\langle \bm{q},\bm{z}\rangle},\quad {\rm with}\quad g_{\bm q}=\sum_{\bm p\in\mathbb Z[{\rm col}(\mathbf{P})]}v_{\bm q-\bm p}u_{\bm p}.
    \end{equation}
    The weak form of Eq. \eqref{eq1} is to find $u\in H_{\rm qp}^1(\mathbb R^d)$, such that
    \begin{equation}\label{eq: weakformqp}
        \dfrac{1}{2}(\nabla u,\nabla w)_{\rm qp}+(g,w)_{\rm qp}=E(u,w)_{\rm qp},\quad\quad \forall w\in H_{\rm qp}^1(\mathbb R^d).
    \end{equation}
    For $u$, one has its generalized Fourier series,
    \begin{equation}
        u(\bm z)=\sum_{\bm q\in \mathbb Z[{\rm col}(\mathbf{P})]}u_{\bm{q}}e^{\mathrm{i}\langle \bm{q},\bm{z}\rangle}.
    \end{equation}
    Then by the orthogonality of  $\{  e^{\mathrm{i}\langle \bm{q},\bm{z}\rangle}\}$, one obtains from Eq. \eqref{eq: weakformqp} that 
    \begin{equation}\label{eq231}
         Eu_{\bm{q}}=\dfrac{1}{2}|\bm{q}|^2u_{\bm{q}}+g_{\bm{q}},\quad \bm q\in \mathbb Z[{\rm col}(\mathbf{P})].
    \end{equation}
Then by the Parseval's identity and the fact that $\|u\|=1$, one has
    \begin{equation}\label{eq36}
        |g_{\bm q}|=\left|\sum_{\bm p\in\mathbb Z[{\rm col}(\mathbf{P})]}v_{\bm q-\bm p}u_{\bm p}\right|\leq \dfrac{1}{2}\sum_{\bm p\in\mathbb Z[{\rm col}(\mathbf{P})]}(v_{\bm q-\bm p}^2+u_{\bm p}^2)=(\|v\|^2+1)/2.
    \end{equation}
    Because $|\bm q|^2>4E$, it is direct to verify that
    \begin{equation}
        |E-|\bm{q}|^2/2|>|\bm{q}|^2/4.
    \end{equation}
    Therefore, by Eqs. \eqref{eq231} and \eqref{eq36}
    \begin{equation}\label{eq-2}
        |u_{\bm{q}}|=\dfrac{|g_{\bm{q}}|}{|E-|\bm{q}|^2/2|}\leq 2(\|v\|^2+1)|\bm q|^{-2}.
    \end{equation}
    By Eq. \eqref{eq231} again, one has
    \begin{equation}\label{coredecom}
        \left|E-\dfrac{1}{2}|\bm q|^2\right||u_{\bm q}|\leq\sum_{\bm p\in\Gamma_1}|v_{\bm q-\bm p}||u_{\bm p}|+\sum_{\bm p\in\Gamma_2}|v_{\bm q-\bm p}||u_{\bm p}|,
    \end{equation}
    where $\Gamma_1=\{\bm p|\bm p\in\mathbb Z[{\rm col}(\mathbf{P})],|\phi^{-1}(\bm q-\bm p)|>\|\mathbf{P}\|^{-1}|\bm q|/2\}$ and $\Gamma_2=\mathbb Z[{\rm col}(\mathbf{P})]\backslash\Gamma_1$. We then estimate the summations in two parts, separately.
    
    $\mathbf {Estimate~of~the \ \Gamma_1\ part}$. By the Parseval's identity and $\|u\|=1$, one can obtain that for any $\bm{q}$, $|u_{\bm{q}}|\leq1$. By the fact that $V(\bm x)\in H_{\rm per}^m([0,2\pi]^n)$ and Eq. \eqref{decay}, one obtains that
    \begin{equation}
        \left|V_{\phi^{-1}(\bm q-\bm p)}\right|\leq \dfrac{n^{m/2}}{(2\pi)^n}\left|\phi^{-1}(\bm q-\bm p)\right|^{-m}|V|_{{\rm per},m}.
    \end{equation}
    Then the bound
     \begin{equation}
        \begin{split}
            \sum_{\bm p\in\Gamma_1}|v_{\bm q-\bm p}||u_{\bm p}|&\leq\sum_{\bm p\in\Gamma_1}|V_{\phi^{-1}(\bm q-\bm p)}|\leq \dfrac{n^{m/2}}{(2\pi)^n}|V|_{{\rm per},m} \sum_{\bm p\in\Gamma_1}|\phi^{-1}(\bm q-\bm p)|^{-m}\\
            &=\dfrac{n^{m/2}}{(2\pi)^n}|V|_{{\rm per},m}\|\mathbf{P}\|^{-1}\sum_{|\bm k|>|\bm q|/2}|\bm k|^{-m}
        \end{split}
     \end{equation}
     holds, where $\bm k\in\mathbb Z^n$. Then by the fact that
     \begin{equation}
        \sum_{|\bm k|>|\bm q|/2}|\bm k|^{-m}\leq\int_{|\bm t|>|\bm q|/2}|\bm t|^{-m}d\bm t=2\pi\cdot2^{n-2}\int_{|\bm q|/2}^{+\infty}|\bm t|^{n-m-1}d|\bm t|=\dfrac{2^{m-1}\pi}{m-n}|\bm q|^{n-m},
    \end{equation}
    one has
    \begin{equation}\label{part1}
        \sum_{\bm p\in\Gamma_1}|v_{\bm q-\bm p}||u_{\bm p}|\leq C_1|\bm q|^{n-m},
    \end{equation}
    where $C_1$ is a positive constant depending on $|V|_{{\rm per},m},\|\mathbf{P}\|^{-1},m$ and $n$.
    
    $\mathbf {Estimate~of~the \ \Gamma_2\ part}$. By the H\"older's inequality with $1/\beta+1/n=1$, 
    \begin{equation}
    \begin{split}
        \sum_{\bm p\in\Gamma_2}|v_{\bm q-\bm p}||u_{\bm p}|&\leq\left(\sum_{\bm p\in\Gamma_2}|v_{\bm q-\bm p}|^2\right)^{\frac{1}{\beta}}\left(\sum_{\bm p\in\Gamma_2}|v_{\bm q-\bm p}|^{\frac{\beta n-2n}{\beta}}|u_{\bm p}|^{n}\right)^{\frac{1}{n}}\\
        &\leq \max_{\bm p\in \Gamma_2}|v_{\bm q-\bm p}|\cdot\|v\|^{\frac{2}{\beta}}\left(\sum_{\bm p\in\Gamma_2}|u_{\bm p}|^{n}\right)^{\frac{1}{n}}
        \leq \|v\|^{\frac{2}{\beta}+\frac{1}{2}}\left(\sum_{\bm p\in\Gamma_2}|u_{\bm p}|^{n}\right)^{\frac{1}{n}}.
    \end{split}
    \end{equation}
    For any $\bm p\in\Gamma_2$, one has 
    \begin{equation}
        \|\mathbf{P}\|^{-1}|\bm q|/2\geq|\phi^{-1}(\bm q-\bm p)|\geq \|\mathbf{P}\|^{-1}|\bm q-\bm p|,
    \end{equation}
    hence $|\bm p|>|\bm q|/2$. Because of Eq. \eqref{eq-2}, one can assume that the decay rate of $|u_{\bm q}|$ with $|\bm q|$ is $\alpha$, that is, there exists a positive constant $C_{\alpha}$ independent with $\bm q$ such that $|u_{\bm{q}}|\leq C_{\alpha}|\bm q|^{-\alpha}$. Then there exists a positive constant $C_2$ depending on $n$ and such that
    \begin{equation}
        \sum_{\bm p\in\Gamma_2}|u_{\bm p}|^{n}\leq C_{\alpha}\sum_{\bm p\in\Gamma_2}|\bm p|^{-\alpha n}\leq C_{\alpha}|\Gamma_2||\bm q/2|^{-\alpha n}\leq C_2 |\bm q|^{n-\alpha n}.
    \end{equation}
    Therefore, there exists a positive constant $C_3$ depending on $n$ and $\|v\|$ such that
    \begin{equation}\label{part2}
        \sum_{\bm p\in\Gamma_2}|v_{\bm q-\bm p}||u_{\bm p}|\leq C_3|\bm q|^{-\alpha+1}.
    \end{equation}
    Combining Eqs. \eqref{coredecom}, \eqref{part1} and \eqref{part2}, one obtains that
    \begin{equation}
        |u_{\bm q}|\leq 4|\bm q|^{-2}\left(\sum_{\bm p\in\Gamma_1}|v_{\bm q-\bm p}||u_{\bm p}|+\sum_{\bm p\in\Gamma_2}|v_{\bm q-\bm p}||u_{\bm p}|\right)\leq 4C_1|\bm q|^{n-m-2}+4C_3|\bm q|^{-\alpha-1}.
    \end{equation} 
    Because $|u_{\bm{q}}|\leq C_{\alpha}|\bm q|^{-\alpha}$ holds, one obtains that $-\alpha\geq n-m-2$, which ends the proof.
\end{proof}

We employ a 1D quasiperiodic problem of \eqref{eq1} as an example to validate the theoretical result of Theorem \ref{thm6}. Let $v(z)=E_0/\big(1+(\cos(z)+\cos(\sqrt{5}z))^2 \big)$ and the projection matrix $\mathbf{P}=[\sqrt{5}\ \ 1]$.  The PM is employed to solve this problem and depict the generalized Fourier coefficients of eigenfunctions in the raised frequency domain. As shown in Figure \ref{fig3-1} (a), whether it is the eigenfunction corresponding to spectrum $0.5945$ (the red line, the smallest spectrum) or the eigenfunction corresponding to spectrum $0.6297$ (the blue line), their generalized Fourier coefficients both decay exponentially. We adopt the RPM to solve the same problem. In order to quantify the truncation error between the PM and RPM, we define 
\begin{equation}\label{epsd}
    \hbox{Err}(D)=\sum_{|\bm k|>D,\bm k\in \Omega}|U_{\bm k}|^2.
\end{equation}
Figure \ref{fig3-1} (b) depict the exponential decay of $\hbox{Err}(D)$ with respect to $D$. For $D\approx 30$, the method has the machine precision and the reduction error becomes negligible, which accounts for more than $80\%$ reduction of the DOF in the eigenvalue solver. 

\begin{figure}[t!]
	\centering
	\includegraphics[width=0.725\textwidth]{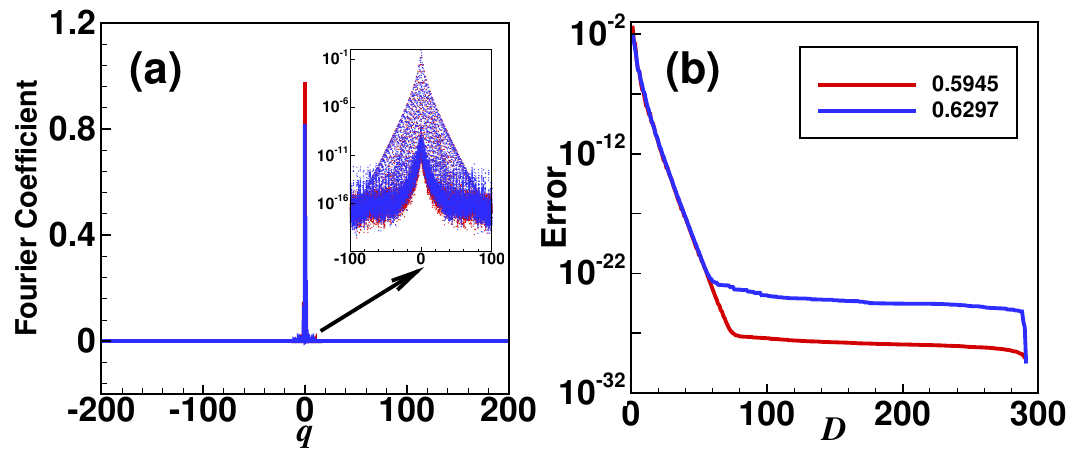}
	\caption{The generalized Fourier coefficient modulus of eigenfunctions for the 1D quasiperiodic potential. (a) The generalized Fourier coefficient modulus of eigenfunctions as function of $\bm q$. (b) The  error $\hbox{Err}(D)$ as function of $D$ for spectrum 0.5945 and 0.6297. In both panels, $E_0=1$ and $N=180$.}\label{fig3-1}
\end{figure}

\subsection{Numerical scheme}
For the PM method, the Fourier approximation space of the numerical solution is given by $\mathbb D={\rm span}\{e^{{\rm i}\langle\bm k,\bm x\rangle},\bm k\in\Omega\}$ with the index set $\Omega$ given by
\begin{equation}\label{omega}
    \Omega=\big\{\bm{k}\in\mathbb{Z}^{n} \big| \|\bm{k}\|_\infty\leq N\big\}.
\end{equation}
As can be observed in Theorem \ref{thm6}, given some mild restrictions on
the regularity of quasiperiodic potential function $v(\bm z)$ and its parent function $V(\bm x)$, the index set $\Omega$ of Fourier expansion can be reduced to
\begin{equation}\label{omega2}
    \Omega_R=\big\{\bm{k}\in\mathbb{Z}^{n} \big| \|\mathbf{P}\bm{k}\|_\infty\leq D, \|\bm{k}\|_\infty\leq N \big\}
\end{equation}
without sacrificing the accuracy of the approximation. Hence, $\Omega$ depends on $N$ and $\Omega_R$ depends on both $N$ and $D$. Here, parameter $D<N$ is a prescribed truncation constant. The RPM for the quasiperiodic problem \eqref{eq1} reads: find non-trivial $U_D^N\in\mathbb D_R={\rm span}\{e^{{\rm i}\langle\bm k,\bm x\rangle},\bm k\in\Omega_R\}$ and $E\in \mathbb R$ such that
\begin{equation}\label{weakform2}
    \dfrac{1}{2}(\varphi(\nabla)U_D^N, \varphi(\nabla)W)+(G_D^N,W)=E(U_D^N,W), \quad W\in\mathbb D_R,
\end{equation}
where $G_D^N=VU_D^N$ and
{\begin{equation}
    \varphi(\nabla)=\left(\sum_{i=1}^dP_{i1}\dfrac{\partial}{\partial x_1},\cdots,\sum_{i=1}^dP_{in}\dfrac{\partial}{\partial x_n}\right).
\end{equation}}

In what follows, we first give a rigorous estimate of the truncation error of the reduced space $\mathbb D_R$ for quasiperiodic functions. For a quasiperiodic function $f$, the operators $\mathcal{P}$ and $\mathcal{Q}$ denote the partial sums 
\begin{equation}
\mathcal{P}f=\sum_{\bm{q}\in\phi(\Omega)}f_{\bm{q}}e^{\mathrm{i}\langle\bm{q},\bm{z}\rangle}=\sum_{\bm{k}\in\Omega}F_{\bm{k}}e^{\mathrm{i}\langle\bm{k},\bm{x}\rangle},
\end{equation}
and
\begin{equation}
\mathcal{Q}f=\sum_{\bm{q}\in\phi(\Omega_R)}f_{\bm{q}}e^{\mathrm{i}\langle\bm{q},\bm{z}\rangle}=\sum_{\bm{k}\in\Omega_R}F_{\bm{k}}e^{\mathrm{i}\langle\bm{k},\bm{x}\rangle},
\end{equation}
where $\Omega$ and $\Omega_R$ are defined in Eqs. \eqref{omega} and \eqref{omega2}. The operator $\mathcal{P}$ depends on $N$ and $\mathcal{Q}$ depends on both $N$ and $D$. In order to obtain the bound of the truncation error, one needs the following lemma (see \cite{shen2011spectral}).

\begin{lemma}\label{lemma2}
    For any quasiperiodic function $f$ with its parent function\\ $F\in H^m_{\rm per}([0,2\pi]^n]),m\in\mathbb{Z}^+$, and $0\leq\mu\leq m$,  the following estimate for $\mathcal{P}f$ holds
    \begin{equation}
        \|\mathcal{P}f-F\|_{\mu,\rm per}\leq N^{\mu-m}|F|_{m,\rm per}.
    \end{equation}
    In addition, if $F\in H^\nu_{\rm per}([0,2\pi]^n])$ with $\nu>n/2$, there exists a constant $C$ depending on $\nu$ such that
    \begin{equation}
        \|\mathcal{P}f-F\|_{\infty,\rm per}\leq CN^{n/2-\nu}|F|_{\nu,\rm per}.
    \end{equation}
\end{lemma}
The truncation error under different norms of operator $\mathcal{Q}$ is bounded, as shown in Theorem \ref{corecorethm1}.

\begin{theorem}\label{corecorethm1}
Suppose that $u$ is a quasiperiodic function. Let $U$ be the parent function of $u$. If $U\in H^m_{\rm per}([0,2\pi]^n),u\in H^{m'}_{\rm qp}(\mathbb R^d)$ with $m,m'\in\mathbb Z^+$ and $0\leq\mu< m\leq m'$, there exist constants $C_1$ and $C_2$ depending on $\|\mathbf{P}\|$ and $\mu$ such that
\begin{equation}
    \|\mathcal{Q}u-u\|_{\mu,\rm qp}\leq C_1N^{\mu-m}|U|_{m,\rm per}+C_2D^{\mu-m'}|u|_{m',\rm qp}.
\end{equation}
If $U\in H^{\nu}_{\rm per}([0,2\pi]^n),u\in H^{\eta}_{\rm qp}(\mathbb R^d)$ with $n/2<\nu\leq\eta$, there exist constants $C_3$ and $C_4$ depending on $\lambda,d,\nu$ and $\eta$ such that
\begin{equation}
    \|\mathcal{Q}u-u\|_{\infty,\rm qp}\leq C_3N^{n/2-\nu}|U|_{\nu,\rm per}+C_4D^{n/2-\eta}|u|_{\eta,\rm qp}.
\end{equation}
\end{theorem}

\begin{proof}
By Lemma \ref{lemma2}, if $U\in H^m_{\rm per}([0,2\pi]^n)$ with $m\in\mathbb{Z}^+$ and $0\leq\mu\leq m$, one has
    \begin{equation}\label{lemma3}
        \|\mathcal{P}u-U\|_{\mu,{\rm per}}\leq N^{\mu-m}|U|_{m,\rm per}.
    \end{equation}
    In addition, if $U\in H_{\rm per}^\nu([0,2\pi]^n)$ with $\nu>n/2$, there exists a constant $C$ depending on $\nu$ satisfying
    \begin{equation}\label{lemma6}
        \|\mathcal{P}u-U\|_{\infty,\rm per}\leq CN^{n/2-\nu}|U|_{\nu,\rm per}.
    \end{equation}
    By a direct decomposition $\mathcal{Q}u-u=(\mathcal{Q}u-\mathcal{P}u)+(\mathcal{P}u-U)$ and the triangle inequality, one has 
    \begin{equation}\label{andy1}
        \|\mathcal{Q}u-u\|_{\mu,\rm qp}\leq\|\mathcal{Q}u-\mathcal{P}u\|_{\mu,\rm qp}+\|\mathcal{P}u-u\|_{\mu,\rm qp},
    \end{equation}
    \begin{equation}\label{andy2}
        \|\mathcal{Q}u-u\|_{\infty,\rm qp}\leq\|\mathcal{Q}u-\mathcal{P}u\|_{\infty,\rm qp}+\|\mathcal{P}u-u\|_{\infty,\rm qp}.
    \end{equation}
    Firstly consider the $\mu$-norm case. For any $\bm{k}\in\Omega/\Omega_R$, $\|\mathbf{P}\bm{k}\|_\infty>D$ holds, so that
    \begin{equation}\label{pfandy1}
    \begin{split}
        \|\mathcal{Q}u-\mathcal{P}u\|_{\mu,\rm qp}^2&\leq D^{-2m'+2\mu}\sum_{\bm{q}\in\phi(\Omega/\Omega_R)}(1+|\bm{q}|^{2\mu})|u_{\bm{q}}|^2|\bm{q}|^{2m'-2\mu}\\
        &\lesssim D^{-2m'+2\mu}|u|_{m',\rm qp}^2,
    \end{split}
    \end{equation}
   where $A\lesssim B$ denotes that $A$ is less than or similar to $B$. Then by Lemma \ref{lemma2}, if $U\in H^m_{\rm per}([0,2\pi]^n)$ with $m\in\mathbb{Z}^+$ and $0\leq\mu\leq m$, one has
    \begin{equation}
        \|\mathcal{P}u-U\|_{\mu,\rm per}\leq N^{\mu-m}|U|_{m,\rm per}.
    \end{equation}
    Then by the definition of the Sobolev spaces $H_{\rm per}^{\mu}$ and $H_{\rm qp}^{\mu}$, one has
    \begin{equation}\label{pfandy3}
        \|\mathcal{P}u-u\|_{\mu,\rm qp}\leq \|\mathbf{P}\|^{\mu}\|\mathcal{P}u-U\|_{\mu,\rm per} .
    \end{equation}
    Combining Eqs. \eqref{pfandy1} and \eqref{pfandy3}, the $\mu$-norm case is proved.
    
On the other hand, one can use the Cauchy-Schwarz inequality to obtain
    \begin{equation}\label{eq333}
        \begin{split}
            \|\mathcal{Q}u-\mathcal{P}u\|_{\infty,\rm qp}\leq \sum_{\bm{q}\in\phi(\Omega/\Omega_R)}|u_{\bm{q}}|\leq \left(\sum_{\bm{q}\in\phi(\Omega/\Omega_R)}|\bm{q}|^{-2\eta}\right)^{\frac{1}{2}}\left(\sum_{\bm{q}\in\phi(\Omega/\Omega_R)}|\bm{q}|^{2\eta}|u_{\bm{q}}|^2\right)^{\frac{1}{2}}.
        \end{split}    
    \end{equation}
    For $\eta>n/2$,
    {\begin{equation}\label{eq334}
    \begin{split}
        \sum_{\bm{q}\in\phi(\Omega/\Omega_R)}|\bm{q}|^{-2\eta}&=\sum_{\bm{k}\in\Omega/\Omega_R}|\mathbf{P}\bm{k}|^{-2\eta}\lesssim N^{n-d}\int_{|\bm z|\geq D} |\bm z|^{-2\eta}d\bm z\\
        &=\dfrac{2^{2\eta-1}\pi}{2\eta-d}\lambda^{n-d}D^{n-2\eta},
    \end{split}
    \end{equation}
    where $C_0$ is a positive constant depending on $\eta$ and $d$.} Hence by Eqs. \eqref{eq333} and \eqref{eq334}, there exists a positive constant $C_4$ depending on $\lambda,d$ and $\eta$ such that
    \begin{equation}\label{pfandy4}
        \|\mathcal{Q}u-\mathcal{P}u\|_{\infty,\rm qp}\leq C_4D^{n/2-\eta}|u|_{\eta,\rm qp}.
    \end{equation}
    {By Lemma \ref{lemma2}, one has
    \begin{equation}\label{pfandy5}
        \|\mathcal{P}u-u\|_{\infty,\rm qp}\leq C_3N^{n/2-\nu}|U|_{\nu,\rm per},
    \end{equation}
    where $C_3$ is a positive constant depending on $\nu$.} Then combining Eqs. \eqref{pfandy4} and \eqref{pfandy5}, the infinite norm case is proved.
\end{proof}

{
\begin{remark}
    If the parent function $F(\bm x)\in H^m_{\rm per}$, one can use Theorem \ref{cosis} to obtain $f(\bm z)\in H^m_{\rm qp}$. This is because
    \begin{equation}
    \begin{split}
        \|f(\bm z)\|_{m,\rm qp}^2&=\sum_{\bm q\in\mathbb Z[{\rm col}(\mathbf{P})]}(1+\|\bm q\|_2^{2m})|f_{\bm q}|^2=\sum_{\bm k\in\mathbb{Z}^n}(1+\|{\mathbf P}\bm k\|_2^{2m})|F_{\bm k}|^2\\&\leq C\|F(\bm x)\|_{m,\rm per},
    \end{split}
    \end{equation}
    where $C$ is a positive constant depending on $\|\mathbf P\|$. Hence $m\leq m'$ and $\nu\leq \eta$ always hold, and the error decay rate of $D$ is not slower than that of $N$. This shows the feasibility of a secondary truncation along the direction of $\mathbf{P}\bm k$. This also implies that, if $u$ and $U$ have sufficiently good regularity, the truncation error can achieve exponential decay with respect to both $N$ and $D$. 
\end{remark}}

\subsection{Solution Algorithm}
In this part, we propose the detailed scheme of the RPM under the weak form Eq. \eqref{weakform2} with the similar spatial notations in \cite{liao2022adaptive}. For the space domain $[0,2\pi]^n$, consider the uniform length $h=2\pi/N$ in each direction for an even positive integer $N$. Define $T_h=\{x=mh,m=0,1,\cdots,N-1\}^n$. For any $U(\bm x)\in H^s_{\rm per}([0,2\pi]^n)$, define $I_N$ as the trigonometric interpolation operator \cite{shen2011spectral},
\begin{equation}
    (I_NU)(\bm x)=\sum_{\bm k\in\Omega_R}\tilde{U}_{\bm k}e^{{\rm i}\langle\bm k,\bm x\rangle},
\end{equation}
where the pseudo-spectral coefficients $\tilde{U}_{\bm k}$ are determined such that $(I_NU)(\bm x)=U(\bm x_h)$ hold for all $\bm x_h\in T_h$. The Fourier pseudo-spectral first and second order derivatives of $(I_NU)(\bm x)$ along the  $x_1$ direction are given by
\begin{equation}
    \mathcal{D}_1(I_NU)=\sum_{\bm k\in\Omega_R}{\rm i}k_1\tilde{U}_{\bm k}e^{{\rm i}\langle\bm k,\bm x\rangle},\quad \mathcal{D}_1^2(I_NU)=-\sum_{\bm k\in\Omega_R}k_1^2\tilde{U}_{\bm k}e^{{\rm i}\langle\bm k,\bm x\rangle}.
\end{equation}
The differentiation operators of other directions can be defined in a similar form. In turn, we can define the discrete higher-dimensional Laplacian in the point-wise sense by
\begin{equation}
    \varphi(\Delta)(I_NU)=\sum_{\bm k\in\Omega_R}\tilde{U}_{\bm k}{\rm tr}(\mathbf{P}\mathbf{H}_{\bm x}\mathbf{P}^{\intercal})e^{{\rm i}\langle\bm k,\bm x\rangle},
\end{equation}
where the operator matrix $\mathbf{H}_{\bm x}$ has the form
\begin{equation}
    \mathbf{H}_{\bm x}=\begin{bmatrix}
         \frac{\partial^2}{\partial x_1^2} & \frac{\partial^2}{\partial x_1\partial x_2} & \cdots & \frac{\partial^2}{\partial x_1\partial x_n}\\
         \frac{\partial^2}{\partial x_2\partial x_1} & \frac{\partial^2}{\partial x_2^2} & \cdots & \frac{\partial^2}{\partial x_2\partial x_n}\\
         \ldots & \ldots & & \ldots\\
         \frac{\partial^2}{\partial x_n\partial x_1} & \frac{\partial^2}{\partial x_n\partial x_2} & \cdots & \frac{\partial^2}{\partial x_n^2}
    \end{bmatrix}.
\end{equation}
By the properties of the orthonormal basis, it leads to a system of equations for each frequency mode $\tilde{U}_{\bm k},\bm k\in\Omega_R$ by the properties of the orthonormal basis,
\begin{equation}\label{eq3}
    E\tilde{U}_{\bm{k}}=\dfrac{1}{2}{\rm tr}(\mathbf{P}\bm k\bm k^{\intercal}\mathbf{P}^{\intercal}) \tilde{U}_{\bm{k}}+\sum_{\bm m\in\Omega_R}\tilde{V}_{\bm k-\bm m}\tilde{U}_{\bm m}.
\end{equation}

One denotes that $\bm{\hat{U}}$ is an $|\Omega_R|\times1$ column vector with its components being the Fourier coefficients $U_{\bm{k}}$, and the discrete eigenvalue problem Eq. \eqref{eq3} can be rewritten into a matrix eigenvalue problem $\mathbf{H}\bm{\hat{U}}=E\bm{\hat{U}}$. In real computations, due to the large size of the dense matrix $\mathbf{H}$, it is not practical to store its elements and the eigenvalue problem is to be solved in a matrix-free manner. That is, for matrix $\mathbf{H}$, one defines its matrix-vector product function
\begin{equation}\label{mvp}
    \mathbf{H}\bm{f}=\mathbf{\hat{D}}\bm{f}+\hbox{FFT}\big(V(\bm x)\cdot\hbox{IFFT}(\bm f)\big),
\end{equation}
where $\hbox{FFT}(\cdot)$ and $\hbox{IFFT}(\cdot)$ denote the $n$-dimensional fast Fourier transform (FFT) and inverse fast Fourier transform (IFFT). Here $\bm{\hat{D}}$ is a diagonal matrix such that for $ \bm{\hat{U}}_m=\tilde{U}_{\bm k}$, the $m$th diagonal element of $\bm{\hat{D}}$ is
\begin{equation}\label{eq350}
\hat{D}_{mm}=\dfrac{1}{2}{\rm tr}(\mathbf{P}\bm k\bm k^{\intercal}\mathbf{P}^{\intercal}).
\end{equation}
Since only the operation $\mathbf{H} \bm f$ is invoked during the generation of basis vectors in the Krylov subspace, there is no need to store the dense matrix $\mathbf{H}$ itself, thus reducing the storage cost from $O(|\Omega_R|^2)$ to $O(|\Omega_R|)$. Then the eigenvalue problem \eqref{eq3} can be solved via the Krylov subspace iterative method in a matrix-free manner \cite{van2000krylov,watkins2007matrix}. To solve the eigenvalues of $\mathbf{H}$ by the Krylov subspace method, one takes a random starting vector $\bm b\in\mathbb{R}^{|\Omega_R|\times1}$, and generates the Krylov subspace $K_M=\hbox{span}\{\bm b,\mathbf{H}\bm b,\cdots,\mathbf{H}^{M-1}\bm b\}$. The orthonormal basis $\bm Q_M=(\bm q_1,\bm q_2,\cdots,\bm q_M)$ of $K_M$ can be generated by the implicitly restarted Arnoldi method \cite{lehoucq2001implicitly}. One can determine the Hessenberg matrix $\mathbf{H}_M=\bm Q_M^{\intercal}\mathbf{H}\bm{Q}_M=\bm Q_M^{\intercal}(\mathbf{H}\bm q_1,\mathbf{H}\bm q_2,\cdots,\mathbf{H}\bm q_M)$ and solve its eigenpairs $\{(E_m,u_m)\}_{m=1}^M$ of $\mathbf{H}_M$ by the QR algorithm. Detailed procedures of the RPM are summarized in Algorithm \ref{algorithm1}. 

\begin{algorithm}[h]
	\caption{The reduced projection method (RPM)}
	\label{algorithm1}
    \leftline{{\bf Input:}~$d$-dimensional quasiperiodic potential $v$, projection matrix $\mathbf{P}$, and parameter}
    \leftline{$N$, $D$, $M$, $\delta$ and $\epsilon$} 
	\leftline{{\bf Output:}~Eigenpairs $\{(E_m,u_m)\}$}
	\begin{algorithmic}[1]
            \State Determine the index sets $\Omega_R$ of basis space and take random starting vector $\bm b\in\mathbb{R}^{2|\Omega_R|\times 1}$ 
            \State Compute $\mathbf D$ by Eq. \eqref{eq350}
            \State Generate the matrix-vector product $\mathbf H\bm b$ by \eqref{mvp}
            \State Repeat step 4-5 to calculate $\mathbf H^2\bm b,\cdots,\mathbf H^{G-1}\bm b$ and generate the Krylov subspace $K_G$ and its orthonormal basis $\mathbf Q_G$ 
            \State Determine Hessenberg matrix $\mathbf H_G$ and compute its eigenpairs $\{(E_m,U_m)\}_{m=1}^G$ of $\mathbf H_G$
            \State Project the eigenfunctions $U_m$ back into the three-dimensional space by $u_m$
	\end{algorithmic}
\end{algorithm}

By the RPM, the DOF and the complexity of the eigenvalue solver are significantly reduced. Specifically, the DOF of the eigenvalue solver for approximating a $d$ dimensional quasiperiodic system with a $d\times n$ projection matrix using the PM is reduced from $O(N^{n})$ to $O(N^{n-d}D^{d})$. Correspondingly, the computational complexity of the proposed algorithm for solving the first $k$ eigenpairs using the Krylov subspace method decreases from $O(kN^{2n})$ to $O(kN^{2(n-d)}D^{2d})$ \cite{lehoucq1998arpack,wright2001large}. It is remarked that the zero-fill operation for the FFT leads
to $O(N^n \log N)$ complexity. This complexity is usually much smaller than
the $O(N^{2(n-d)}D^{2d})$ scaling, and the FFT calculation is only a small portion in the eigenvalue solver. Due to the fast decay of the generalized Fourier coefficients along $\mathbf{P}\bm{k}$, the RPM can thus obtain reliable numerical results with much fewer DOFs of the eigenvalue solver, thereby mitigating the curse of dimensions, especially for high-dimensional problems. 

In what follows, we first give a rigorous estimate of the truncation error of the reduced space for quasiperiodic functions. Let $\mathcal{I}f$ represent the trigonometric interpolation associated with the RPM. Because of the use of the FFT and IFFT, the error analysis is quite different. The upper bound of the approximation under different norms of operator $\mathcal{Q}$ to the eigenspace of Eq. \eqref{eq1} is given in Theorem \ref{corecorethm}.

{\begin{theorem}\label{corecorethm}
Suppose that $u$ is the solution of the quasiperiodic Schr\"odinger eigenvalue problem Eq. \eqref{eq1}. Let $U$ be the parent function of $u$. If $U\in H^m_{\rm per}([0,2\pi]^n)$,\\$u\in H^{m'}_{\rm qp}(\mathbb R^d)$ with $m,m'\in\mathbb Z^+$ and $0\leq\mu< m\leq m'$, there exist constants $C_1$ and $C_2$ depending on $\|\mathbf{P}\|$ and $\mu$ such that
\begin{equation}
    \|\mathcal{I}u-u\|_{\mu,\rm qp}\leq C_1N^{\mu-m}|U|_{m,\rm per}+C_2D^{\mu-m'}|u|_{m',\rm qp}.
\end{equation}
If $U\in H^{\nu}_{\rm per}([0,2\pi]^n),u\in H^{\eta}_{\rm qp}(\mathbb R^d)$ with $n/2<\nu\leq\eta$, there exist constants $C_3$ and $C_4$ depending on $\lambda,d,\nu,n$ and $\eta$ such that
\begin{equation}
    \|\mathcal{I}u-u\|_{\infty,\rm qp}\leq C_3N^{n/2-\nu}|U|_{\nu,\rm per}+C_4D^{n/2-\eta}|u|_{\eta,\rm qp}.
\end{equation}
\end{theorem}}

\begin{proof}
    {Because of the discrete orthogonality condition
\begin{equation}
    (e^{\mathrm{i}\langle\bm{k}_1,\bm{x}\rangle},e^{\mathrm{i}\langle\bm{k}_2,\bm{x}\rangle})=\begin{cases}
        1, \quad \bm k_1=\bm k_2+2N \bm m,\bm m\in\mathbb Z^n,\\
        \\
        0,\quad {\rm otherwise},
    \end{cases}
\end{equation}
one has
\begin{equation}
    \mathcal{I}f=\mathcal{Q}f+\mathcal{R}f,
\end{equation}
where 
\begin{equation}
    \mathcal{R}f=\sum_{\bm k\in\Omega_R}\left(\sum_{\bm m\in\mathbb Z^n\backslash\{\bm 0\}}F_{\bm k+2N\bm m}\right)e^{\mathrm{i}\langle\bm{k},\bm{x}\rangle}
\end{equation}
is the aliasing error. Then one has the decomposition $\mathcal{I}u-u=(\mathcal{Q}u-\mathcal{P}u)+(\mathcal{P}u-u)+\mathcal{R}u$. By the triangle inequality, one has 
    \begin{equation}\label{loveandy1}
        \|\mathcal{I}u-u\|_{\mu,\rm qp}\leq\|\mathcal{Q}u-\mathcal{P}u\|_{\mu,\rm qp}+\|\mathcal{P}u-u\|_{\mu,\rm qp}+\|\mathcal{R}u\|_{\mu,\rm qp},
    \end{equation}
    \begin{equation}\label{loveandy2}
        \|\mathcal{I}u-u\|_{\infty,\rm qp}\leq\|\mathcal{Q}u-\mathcal{P}u\|_{\infty,\rm qp}+\|\mathcal{P}u-u\|_{\infty,\rm qp}+\|\mathcal{R}u\|_{\infty,\rm qp}.
    \end{equation}
One can obtain that there exist constants $C_1$ and $C_2$ depending on $\mu$, such that  \cite{canuto2007spectral}
    \begin{equation}\label{pfandy2}
        \|\mathcal{R}u\|_{\mu,\rm per}\leq C_1N^{\mu}\|\mathcal{R}u\|_{L^2}\leq C_2N^{\mu-s}|U|_{s,\rm per}.
    \end{equation}
    Then
    \begin{equation}
        \|\mathcal{R}u\|_{\mu,\rm qp}\leq \|\mathbf{P}\|^{\mu}\|\mathcal{R}u\|_{\mu,\rm per}
    \end{equation}
    holds. In addition, applying the Cauchy-Schwarz inequality, there exists a constant $C_3$ depending on $\nu$ and $n$ such that
    \begin{equation}\label{pfandy6}
        \|\mathcal{R}u\|_{\infty,\rm per}\leq\sum_{\bm k\in\Omega_R}\sum_{\bm m\in\mathbb Z^n\backslash\{\bm 0\}}|U_{\bm k+2N\bm m}|\leq \sum_{\bm k\in \Omega^C}|U_{\bm k}|\leq C_3N^{n/2-\nu}|U|_{\nu,\rm per}.
    \end{equation}
   By using Theorem \ref{corecorethm1}, both cases are proved.}
\end{proof}

\section{Numerical examples}\label{s4}

We present numerical results to demonstrate the effectiveness of the RPM. Specifically, we apply the RPM to quasiperiodic Schr\"{o}dinger eigenvalue problems for moir\'e lattices in 1D and 2D spaces and assess the quality of the resulting solutions, as well as the CPU time.  A matrix-free preconditioned Krylov subspace method \cite{liesen2013krylov,kelley1995iterative,saad2003iterative,stewart2002krylov} is employed which requires only the matrix-vector product to be stored at each iteration. In these examples, we compare the RPM with the PM, which shows the accuracy and efficiency of the RPM. The calculations presented in this section are executed using Matlab code on an Intel TM core with a clock rate of 2.50 GHz and 32 GB of memory.

\subsection{1D example}
We first examine the performance of the RPM for the 1D case. {To be specific, we adopt the potential function in Eq. \eqref{eq1}  to be those for 1D moir\'e lattices \cite{gao2023pythagoras}, expressed by}
\begin{equation}
    v_1(z)=\dfrac{E_0}{\big[\cos\big(2\cos(\theta/2)z\big)+\cos\big(2\sin(\theta/2)z\big)\big]^2+1},
\end{equation}
with $\theta=\pi/6$. The projection matrix is $\mathbf{P}=[2\cos(\theta/2),\ \ 2\sin(\theta/2)]$.

We take $E_0=1$. We fix the number of Fourier expansions to be $N=180$ and depict in Figure \ref{fig3-1-2} the DOF of the eigenvalue solver and condition numbers of $\mathbf{H}$ in Eq. \eqref{mvp} against the truncation parameter $D$. It can be observed clearly that both the DOF and condition numbers decrease rapidly with the decrease of $D$. For comparison, the DOF of the original PM is $N^2=32400$, much bigger than that of the RPM for small $D$.  Thus, a small $D$ not only leads to a matrix eigenvalue system of much smaller size, but also reduces the number of iterations to converge. These observations highlight the potential of using the RPM to solve high-dimensional quasiperiodic eigenvalue problems.

\begin{figure}[t!]
	\centering
	\includegraphics[width=0.725\textwidth]{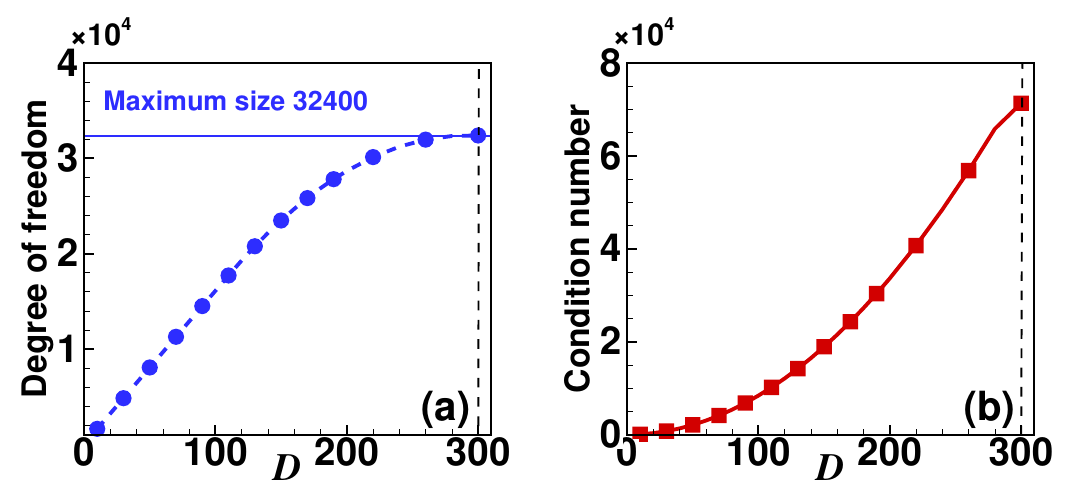}
	\caption{The DOF (a)  and the condition number (b) as function of $D$ using the RPM in one dimension with $N=180$. Correspondingly, the DOF of the PM is $N^2=32400$.}\label{fig3-1-2}
\end{figure}

\begin{figure}[t!]
	\centering
	\includegraphics[width=0.725\textwidth]{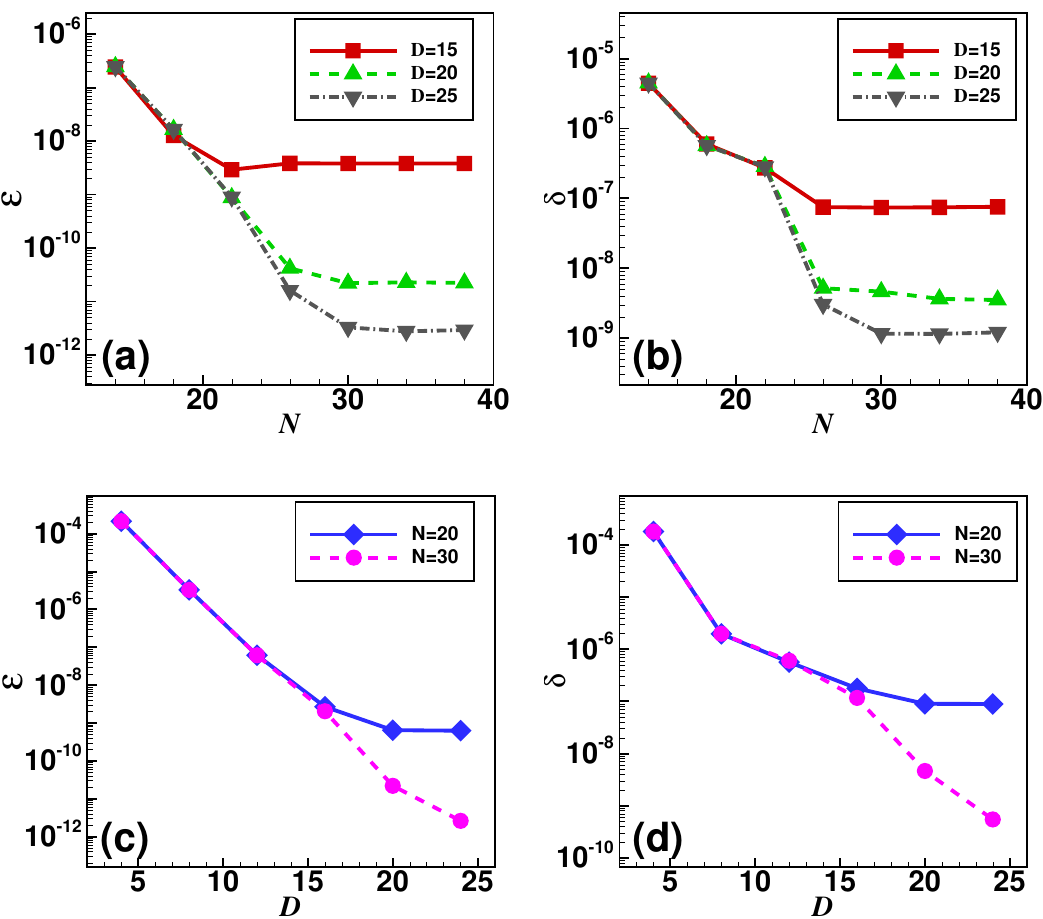}
	\caption{Absolute error of eigenvalues $\varepsilon$ and the $L^2$-error of the first eigenfunction $\delta$. (a,b): Error as function of $N$ for different $D$. (c,d): Error as function of $D$ for different $N$. }\label{fig3-3-1}
\end{figure}

To demonstrate the exceptional accuracy and rapid convergence of the RPM approach, we present error plots in Figure \ref{fig3-3-1} for the potential function with $E_0=1$. The ``exact" eigenvalues and eigenfunctions are determined using the numerical results obtained from the PM when $N=300$. The error of eigenvalue, $\varepsilon$, is measured by the absolute error of the first eigenvalue. The error of the first normalized eigenfunction is also measured by the $L^2$-norm in interval $[0,1]$, which is denoted by $\delta$. Fig. \ref{fig3-3-1} illustrates the convergence with the increase of $N$ for $D=15, 20$ and $25$, and the convergence with the increase of $D$ for $N=20$ and $30$, characterized by both $\varepsilon$ and $\delta$. Panels (a,b) illustrate that both $\varepsilon$ and $\delta$ exhibit an exponential decrease as $N$ increases, eventually attaining a fixed value. It is notable that the magnitude of this fixed value diminishes with the increase of $D$. For relatively small values of $N$ and $D$ (e.g. $N=30$ and $D=20$), $\varepsilon$ is already smaller than $10^{-10}$ and $\delta$ is smaller than $10^{-8}$, demonstrating the high accuracy of the RPM. Panels (c,d) exhibit exponential decays with the increase of $D$, which is in agreement with the error analysis. When $D$ is small ($D\leq15$), the error curves of $N=20$ and $30$ almost overlap, indicating that the error mainly comes from the basis reduction. When $D$ is large, the error curves of the two cases have a significant difference, indicating that the error is mainly caused by the PM part. Overall, one can observe that high accuracy of the results is remained in spite of a significant reduction in the number of bases.  

In Table \ref{tb1}, we display the DOF in the RPM and the CPU time for the 1D system with $E_0=1$ for $N=50, 100$ and $150$
with $D$ increasing from $10$ to $50$. The DOF increases linearly with $D$. Theoretically, the RPM of 1D systems has complexity $O(D^2)$ for given $N$, and $O(N^2)$ for given $D$. Correspondingly, the complexity of the original PM is $O(N^4)$. Moreover, the condition number of the RPM is much smaller than the PM, as shown in Fig. \ref{fig3-1-2}. The results of the CPU time validate the complexity analysis. We have shown that a small $D$ can achieve high accuracy. At $N=50$, setting $D=20$ has error as small as $10^{-10}$. In this case, the CPU time for the RPM is $1.46$ seconds, $11.6$ times faster than that of the original PM. The reduction for large $N$ is more significant. For $N=150$, the speedup with $D=20$ reaches $317.0$ times. Correspondingly, with $N=50,100$ and $150$, the DOF of the original PM  for $D=20$ are $2.4,4.8$ and $7.2$ times greater than those of the RPM, respectively. These results clearly demonstrate the attractive performance of the RPM.

\begin{table}[t!]
	\centering
	\fontsize{7}{7}\selectfont
	\begin{threeparttable}
		\caption{The DOF and CPU time of the RPM for different $N$ and $D$} 
		\label{tb1}
		\begin{tabular}{cccccccc}
			\toprule
			\midrule
            \multirow{3}{*}{} & \multirow{3}{*}{$D$}  & \multicolumn{2}{c}{$N=50$}  & \multicolumn{2}{c}{$N=100$}  & \multicolumn{2}{c}{$N=150$} \cr
                  \cmidrule(r){3-4} \cmidrule(r){5-6} \cmidrule(r){7-8}
                  \noalign{\smallskip} 
                  &  & DOF & CPU time (s) & DOF & CPU time (s) & DOF & CPU time (s)\cr
			\midrule
                PM & - & 2500 & 17.010 & 10000 & 320.604 & 22500 & 3179.439 \cr
                \midrule
                \multirow{9}{*}{\ \ RPM \ \ }& 50 & 2373 & 9.382 & 5177 & 35.500 & 7766 & 123.172 \cr
                & 45 & 2236 & 8.389 & 4659 & 21.791 & 6990 & 85.054 \cr
                & 40 & 2048 & 5.690 & 4141 & 17.595 & 6210 & 61.924 \cr
                & 35 & 1811 & 4.433 & 3623 & 11.190 & 5434 & 42.748 \cr
                & 30 & 1552 & 3.110 & 3106 & 9.733 & 4658 & 30.751 \cr
                & 25 & 1295 & 2.266 & 2587 & 5.404 & 3881 & 18.599 \cr
                & 20 & 1034 & 1.460 & 2069 & 3.529 & 3106 & 10.029 \cr
                & 15 & 777 & 1.000 & 1551 & 2.456 & 2328 & 6.542 \cr
                & 10 & 517 & 0.452 & 1035 & 1.050 & 1552 & 2.800 \cr
			\midrule
			\bottomrule
		\end{tabular}
	\end{threeparttable}
\end{table}

Figure \ref{fig3-3-3} depicts the error of the normalized first eigenfunction in interval $z\in[0, 1]$. We take $D=25$ for $N=20, 40$ and $60$, and calculate the absolute error for different $E_0=1, 2, 4$ and $8$ where the ``exact" eigenfunctions are generated by using the numerical results of the PM with $N=300$. One can observe that the error converges with the increase of $N$. With the increase of $E_0$, the error of the RPM increases. This is because $E_0$ describes the optical response strength  in the photorefractive crystal \cite{wang2020localization}. For a large $E_0$, the eigenfunction can become localized, leading to an obvious singularity. The results in panels (cd) illustrate that the RPM remains high accuracy with a small $D$ with $N=60$,
demonstrating that the RPM is efficient for simulating challenging problems such as localization-delocalization transition in photonic moir\'e lattices.

\begin{figure}[t!]
	\centering
	\includegraphics[width=0.725\textwidth]{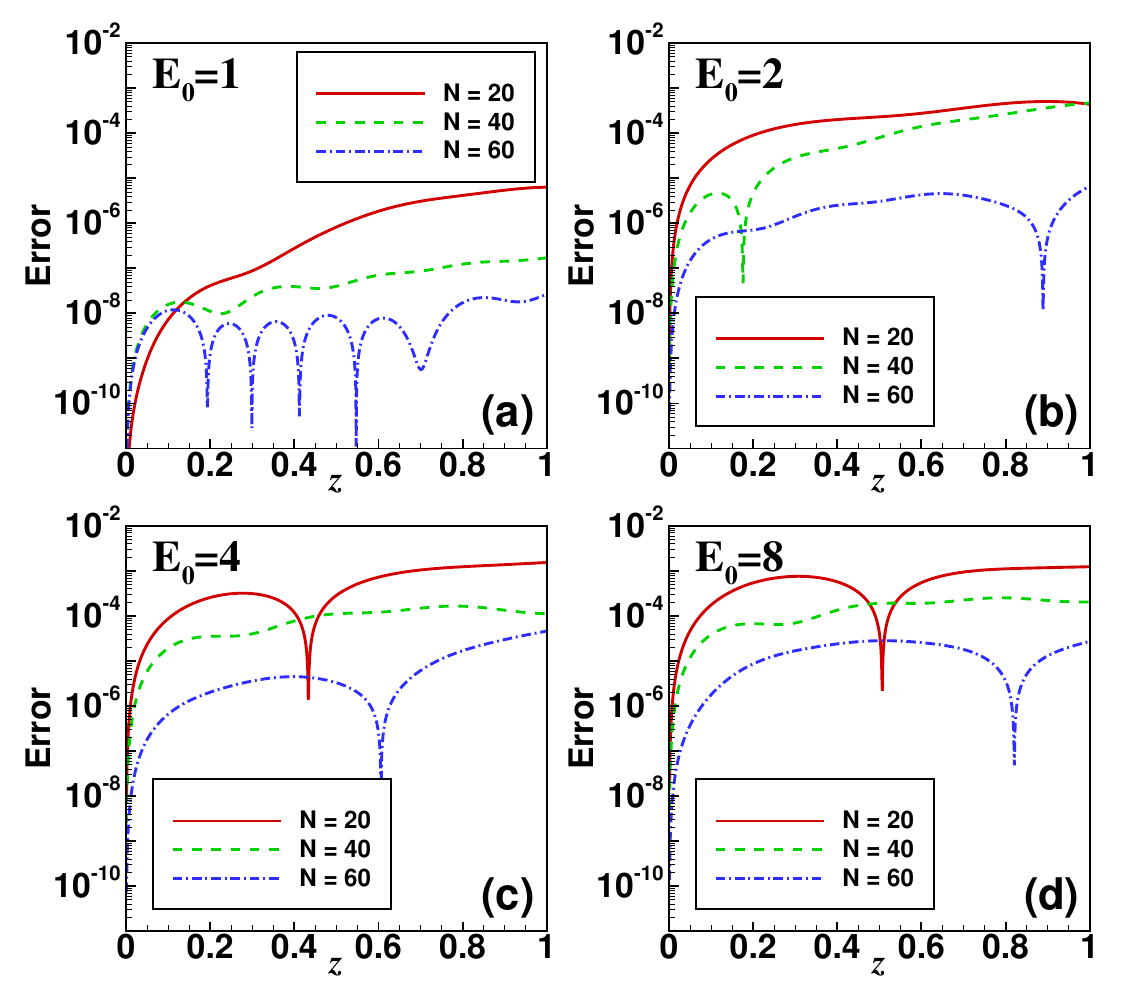}
	\caption{Error of the normalized first eigenfunctions obtained by the RPM in interval $[0, 1]$ for different $E_0$.  In each panel, $D=25$ and three different $N$ are calculated.}\label{fig3-3-3}
\end{figure}

\subsection{2D example}
Consider a 2D  example with the potential function taking 
\begin{equation}\label{eqp}
    v_2(z_1,z_2)=\dfrac{E_0}{(\cos z_1\cos z_2+\cos(\sqrt{5}z_1)\cos(\sqrt{5}z_2))^2+1}.
\end{equation}
This potential possesses the same structure as  2D moir\'e lattices \cite{wang2020localization,gao2023pythagoras}, making it applicable to simulations of photonic lattices.
Correspondingly, the projection matrix is given by,
\begin{equation}
    \mathbf{P}=\begin{bmatrix}
    1 & 0 & \sqrt{5} & 0\\
    0 & 1 & 0 & \sqrt{5}
    \end{bmatrix}.
\end{equation}

\begin{figure}[t!]
	\centering
	\includegraphics[width=0.725\textwidth]{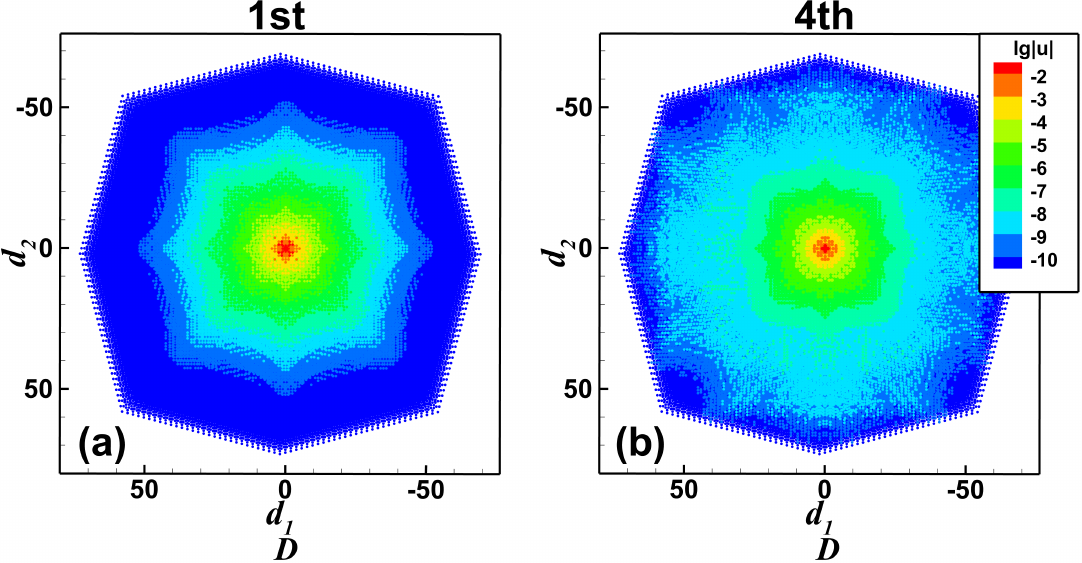}
	\caption{The generalized Fourier coefficients (modulus) of eigenfunctions in the $\bm{q}$ space for $E_0=1$ and $N=30$. Results are present by logarithms with base $10$. }\label{fig3-2}
\end{figure}

We first calculate the generalized Fourier coefficients of eigenfunctions of the system. We set $E_0=1$ and $N=30$. In Figure \ref{fig3-2}, we show the modulus of the coefficients for the 1st and 4th eigenfunctions in the $\bm{q}$ space, calculated by the PM. The data are present with values of logarithms 10. One can observe the exponential decay of the generalized Fourier coefficients for both eigenfunctions. In each panel, there is only one peak in the origin of the $\bm{q}$ space, and far from the origin the contributions of the Fourier modes are insignificant. {Table \ref{tb3} presents the error measured by $\hbox{Err}(D)$ in Eq. \eqref{epsd} of eigenfunctions corresponding to specturm $0.4961$ (the smallest one) and $0.4975$ with the truncation constant $D$ for $N=30$.} Again, one can observe rapid decays with respect to $D$ for both cases. These results are similar to  the 1D case and demonstrate that the approximation in the  reduced space can be of high accuracy for the eigenproblem.  

\begin{table}[htbp]
	\centering
	\fontsize{8}{8}\selectfont
	\begin{threeparttable}
		\caption{Error of eigenfunctions as function of $D$ in two dimensions} 
		\label{tb3}
		\begin{tabular}{cccccc}
			\toprule
			\midrule
                \multirow{3}{*}{\ \ \ \ $D$ \ \ \ \ } & \multicolumn{2}{c}{$\hbox{Err}(D)$}  & \multirow{3}{*}{\ \ \ \ $D$ \ \ \ \ }  & \multicolumn{2}{c}{$\hbox{Err}(D)$} \cr
                  \cmidrule(r){2-3} \cmidrule(r){5-6}
                  \noalign{\smallskip}
                      & \ \ \ \   0.4961 \ \ \ \   &  \ \ \ \  0.4975 \ \ \ \   &  &  \ \ \ \  0.4961 \ \ \ \    &  \ \ \ \ 0.4975 \ \ \ \  \cr
			\midrule
			5 & 1.54E-05 & 2.19E-05 & 35 & 9.93E-15 & 1.58E-13\cr
                10 & 1.19E-07 & 1.16E-07 & 40 & 6.36E-16 & 7.36E-14\cr
                15 & 2.66E-09 & 2.50E-09 & 45 & 4.76E-17 & 3.40E-14\cr
                20 & 6.40E-11 & 7.30E-11 & 50 & 3.79E-18 & 1.52E-14\cr
                25 & 2.91E-12 & 3.20E-12 & 55 & 2.88E-19 & 2.68E-15\cr
                30 & 1.61E-13 & 3.71E-13 & 60 & 2.23E-20 & 1.73E-16\cr
			\midrule
			\bottomrule
		\end{tabular}
	\end{threeparttable}
\end{table}

Figure \ref{fig3-2-2} presents the DOF and condition number of the RPM as function of $D$ with same setup: $E_0=1$ and $N=30$. Again, both the DOF and condition number increase rapidly with the increase of $D$. In spite that $N=30$ is not big, the DOF of the entire system in the raised 4D space, $N^4=810000$, is a very big number. From the results, we can see that the use of a small $D$ can significantly reduce computational complexity, not only the size of the matrix eigenvalue problem, but also the iteration number in the solver of the implicitly restarted Arnoldi method.

\begin{figure}[ht]
	\centering
	\includegraphics[width=0.725\textwidth]{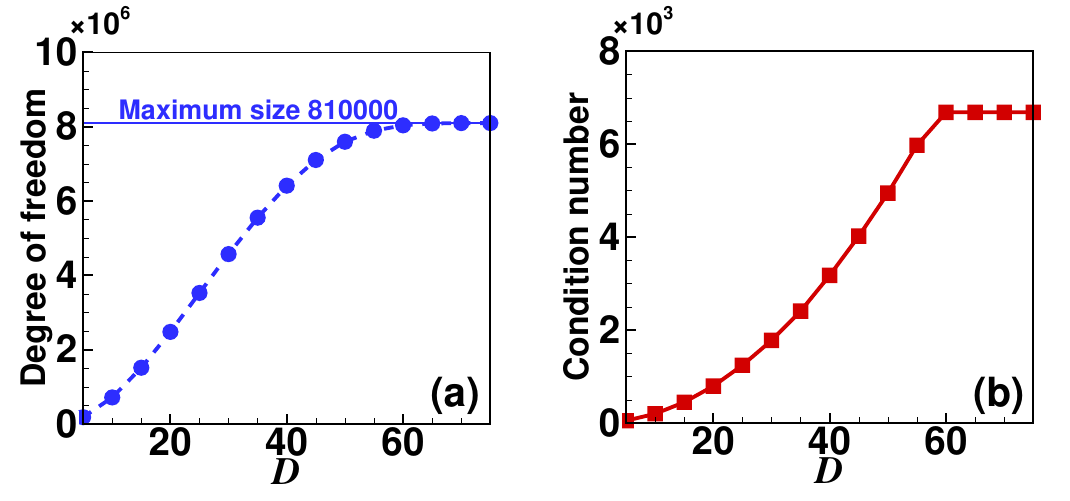}
	\caption{The DOF (a) and the change of condition number (b) with respect to $D$ of the RPM in 2D. $N=30$ and the maximum DOF is $N^4=810000$.}\label{fig3-2-2}
\end{figure}

We next study the accuracy and convergence of the RPM with $E_0=1$, and the results are presented in Figure \ref{fig4-3-1}. In the calculations, the ``exact" eigenvalues and eigenfunctions are obtained from numerical results of the PM with $N=32$. The errors are measured, where $\varepsilon$ represents the absolute error of the first eigenvalue, and $\delta$ represents the error of the first eigenfunction using the $L^2$ norm in interval $[0,1]^2$. Panels (a,b) illustrate the convergence of the  numerical solution with the increase of $N$ for truncation coefficent  $D=10,20$ and $30$. Both $\varepsilon$ and $\delta$ decay exponentially with increasing $N$,  eventually converging to a fixed value which depends on $D$. Similar to the 1D case, small value of $D$ results in high accurate results. For $D=10$ (with a slightly bigger $N$), the RPM can achieve accuracy at the level of $10^{-5}$ in both the eigenvalue and eigenfunction calculation.
Panels (c,d) displays the convergence with the increase of $D$ for $N=20$ and $28$. One can observe  the exponential decays with $D$ at the beginning, as expected from the previous error analysis. For small  $D$, the error curves for $N=20$ and $28$ almost overlap, suggesting that the reduction of the basis space dominates the error. With a mediate $D$, the two curves in both panels differ significantly, indicating that the error at this point mainly comes from the PM part. Overall, the accuracy with small $D$ (e.g., $D=10$) is good enough to provide accurate solutions. These findings demonstrate that high accuracy can be maintained by a significant reduction for basis functions.

\begin{figure}[ht]
	\centering
	\includegraphics[width=0.725\textwidth]{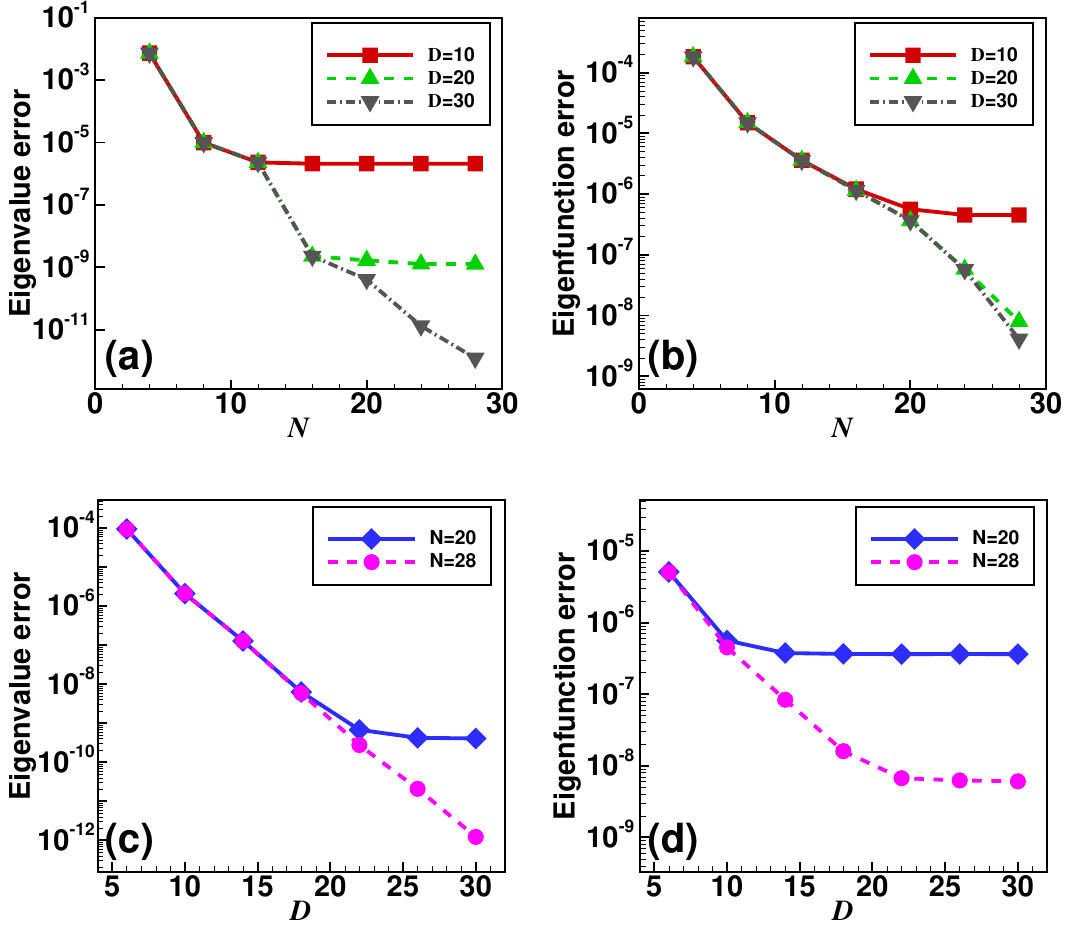}
	\caption{ Errors of the first eigenvalue and eigenfunction. (a,b): Error as function of $N$ for different $D$. (c,d): Error as function of $D$ for different $N$.}\label{fig4-3-1}
\end{figure}

We next conduct the study on the DOF and CPU time required for the RPM for the 2D system with $E_0=1$ for $N=20, 24$ and $28$ with $D$ increasing from $10$ to $30$, and the results are summarized in Table \ref{tb2}. The DOF exhibits a quadratic decrease with respect to $D$. Theoretically, the RPM of 2D systems has complexity $O(D^4)$ for given $N$,  and $O(N^4)$ for given $D$, while the complexity of the original PM is $O(N^8)$. The numerical results of Table \ref{tb2} are in agreement with these theoretical analysis. It also can be found that a small $D$ is able to reach high accuracy. For $N=20$, the use of $D=10$ achieves an error level of $10^{-5}$. In this case, the CPU time for the RPM is $148.50$ seconds which is $15.8$ times faster than that of the PM, and the DOF is $4.9$ times smaller than that of the PM. The reduction for larger $N$ will be more significant. When $N=28$, the speedup with $D=10$ becomes $73.8$ times for the CPU time, and the reduction in the DOF is $9.7$ times, comparing the RPM with the PM. One can see this speedup is even more larger than 1D problems by introducting the reduction technique in the PM.  

\begin{table}[htbp]
	\centering
	\fontsize{7}{7}\selectfont
	\begin{threeparttable}
		\caption{The DOF and CPU time of the RPM for different $N$ and $D$ in two dimensions} 
		\label{tb2}
		\begin{tabular}{cccccccc}
			\toprule
			\midrule
            \multirow{3}{*}{} & \multirow{3}{*}{$D$} & \multicolumn{2}{c}{$N=20$}  & \multicolumn{2}{c}{$N=24$}  & \multicolumn{2}{c}{$N=28$} \cr
                  \cmidrule(r){3-4} \cmidrule(r){5-6} \cmidrule(r){7-8}
                  \noalign{\smallskip} 
                   & & Size & CPU time (s) & Size & CPU time (s) & Size & CPU time (s)\cr
			\midrule
            PM& - & 160000 & 2343.2 & 331776 & 10305 & 614656 & 32901 \cr
			\midrule
			\multirow{11}{*}{RPM}& 30 & 156816 & 2089.2 & 291600 & 6681.5 & 459684 & 15591 \cr
			& 28 & 152100 & 1927.7 & 273529 & 5760.1 & 421201 & 11562\cr
   			& 26 & 145161 & 1803.2 & 252004 & 4807.6 & 379456 & 9889.5\cr
      		& 24 & 135424 & 1552.9 & 227529 & 4318.9 & 335241 & 7631.6\cr
         	& 22 & 123201 & 1284.0 & 200704 & 3144.2 & 291600 & 6118.3\cr
                & 20 & 108900 & 1057.3 & 173889 & 2362.0 & 247009 & 4035.2\cr
                & 18 & 94249 & 786.30 & 145924 & 1895.6 & 202500 & 3363.2\cr
                & 16 & 78400 & 650.76 & 117649 & 1409.2 & 160801 & 1832.3\cr
                & 14 & 62001 & 424.71 & 90601 & 821.09 & 123904 & 1341.1\cr
                & 12 & 46225 & 319.53 & 66564 & 493.80 & 91204 & 763.4\cr
                & 10 & 32400 & 148.50 & 46656 & 304.35 & 63504 & 446.09\cr
                \midrule
                
			\bottomrule
		\end{tabular}
	\end{threeparttable}
\end{table}

Finally, we investigate the performance of the RPM for varying $E_0$. Figure \ref{fig4-3-3} shows the profiles of the first eigenfunctions in 2D quasiperiodic systems for various $E_0$ and $N$. With the increase of $E_0$, the eigenfunction becomes singular, leading to a localized eigenstate. This phenomenon is reminiscent of the localization-delocalization transition exhibited in experimental studies of 2D photonic moir\'e lattices \cite{wang2020localization}. The moir\'e lattices rely on flat-band structures for wave localization as opposed to the disordered media used in other approaches based on light diffusion in photonic quasicrystals \cite{freedman2006wave,levi2011disorder}. The localization-delocalization transition of the eigenstates in 2D systems provide valuable insight into the exploration of quasicrystal structures. This transition process is displayed in Figure \ref{fig4-3-3}. The figure also illustrates that the results of different $N$ are basically the same for the four different $E_0$, indicating that the RPM converges fast for cases of both low and strong strength of optical response. Moreover, because of the lower DOF of the RPM, one can expect that more numerical nodes in each dimension can be applied to achieve higher accuracy of approximation.

\begin{figure}[ht]
	\centering
	\includegraphics[width=0.85\textwidth]{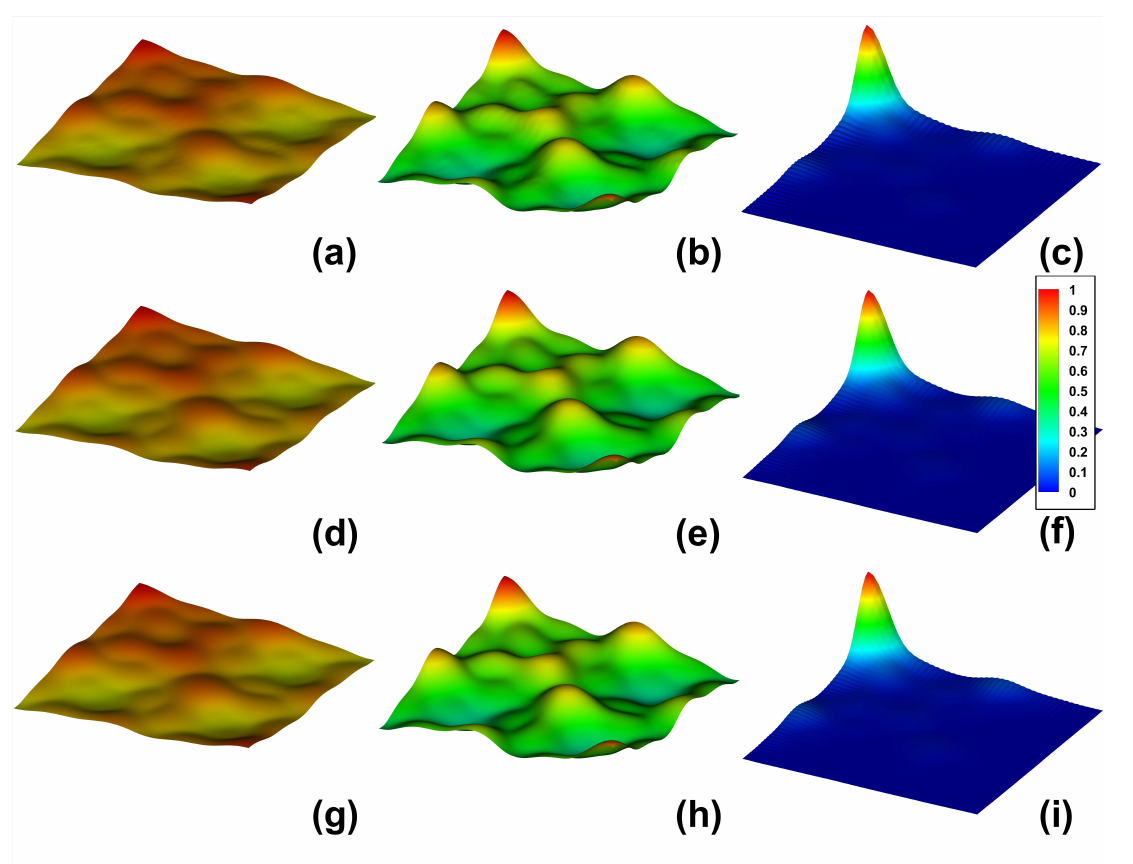}
	\caption{The normalized first eigenfunction $|u|$ in 2D under different $E_0$ and $N$. $D=30$ is taken and the results in area $[0,10]^2$. (a,b,c): $N=24$, (d,e,f): $N=26$ and (g,h,i): $N=28$. (a,d,g): $E_0=0.25$, (b,e,h): $E_0=1$ and (c,f,i): $E_0=4$.}\label{fig4-3-3}
\end{figure}

\section{Conclusions}\label{s5}
{We have proposed a reduced projection method (RPM) for accurate and fast calculations of eigenvalue problems for photonic quasicrystals. We show that the fast decay of the generalized Fourier coefficients for eigenfunctions of the quasiperiodic problems,} justifying the efficiency of the RPM. The error bound of the approximation is provided, which demonstrates the high accuracy from theoretical point of view. Compared to the original PM, the reduced method requires much less memory and significantly speeds up the calculation, making it possible to calculate high-dimensional quasiperiodic eigenvalue problems. Numerical results in both 1D and 2D problems show the efficiency and accuracy of the algorithm, demonstrating attractive features for a broader applications for practical problems. The RPM is potentially useful to solve 3D or other problems rather than the Schr\"odinger system, which will be reported in our future work.

\begin{acknowledgements}
Z. G. and Z. X. are supported by the National Natural Science Foundation of China (NNSFC)(grants No. 12325113 and 12071288), Science and Technology Commission of Shanghai Municipality (grant No. 21JC1403700), and the HPC center of Shanghai Jiao Tong University. Z. Y. is supported by the NNSFC (No. 12101399) and the Shanghai Sailing Program (No. 21YF1421000).  The authors also thank Professor Huajie Chen from Beijing Normal University for helpful discussions.
\end{acknowledgements}

\bibliographystyle{siamplain}
\bibliography{RPMfi}

\begin{thebibliography}{10}

\bibitem{adams2003sobolev}
{\sc R.~Adams and J.~Fournier}, {\em Sobolev spaces}, Elsevier, 2003.

\bibitem{bistritzer2011moire}
{\sc R.~Bistritzer and A.~MacDonald}, {\em Moir{\'e} bands in twisted double-layer graphene}, Proceedings of the National Academy of Sciences, 108 (2011), pp.~12233--12237.

\bibitem{bohr2018almost}
{\sc H.~Bohr}, {\em Almost Periodic Functions}, Courier Dover Publications, 2018.

\bibitem{canuto2007spectral}
{\sc C.~Canuto, M.~Hussaini, A.~Quarteroni, and T.~Zang}, {\em Spectral methods: Fundamentals in Single Domains}, Springer Science \& Business Media, 2007.

\bibitem{cao2018unconventional}
{\sc Y.~Cao, V.~Fatemi, S.~Fang, K.~Watanabe, T.~Taniguchi, E.~Kaxiras, and P.~Jarillo-Herrero}, {\em Unconventional superconductivity in magic-angle graphene superlattices}, Nature, 556 (2018), pp.~43--50.

\bibitem{davenport1946simultaneous}
{\sc H.~Davenport and K.~Mahler}, {\em Simultaneous {D}iophantine approximation}, Duke Mathematical Journal, 13 (1946), pp.~105--111.

\bibitem{freedman2006wave}
{\sc B.~Freedman, G.~Bartal, M.~Segev, R.~Lifshitz, D.~Christodoulides, and J.~Fleischer}, {\em Wave and defect dynamics in nonlinear photonic quasicrystals}, Nature, 440 (2006), pp.~1166--1169.

\bibitem{fu2020optical}
{\sc Q.~Fu, P.~Wang, C.~Huang, Y.~Kartashov, L.~Torner, V.~Konotop, and F.~Ye}, {\em Optical soliton formation controlled by angle twisting in photonic moir{\'e} lattices}, Nature Photonics, 14 (2020), pp.~663--668.

\bibitem{gao2023pythagoras}
{\sc Z.~Gao, Z.~Xu, Z.~Yang, and F.~Ye}, {\em Pythagoras superposition principle for localized eigenstates of two-dimensional moir{\'e} lattices}, Physical Review A, 108 (2023), p.~013513.

\bibitem{goldman1993quasicrystals}
{\sc A.~Goldman and R.~Kelton}, {\em Quasicrystals and crystalline approximants}, Reviews of Modern Physics, 65 (1993), p.~213.

\bibitem{grafakos2008classical}
{\sc L.~Grafakos}, {\em Classical Fourier Analysis}, vol.~2, Springer, 2008.

\bibitem{hu2020moire}
{\sc G.~Hu, A.~Krasnok, Y.~Mazor, C.~Qiu, and A.~Al{\`u}}, {\em Moir{\'e} hyperbolic metasurfaces}, Nano Letters, 20 (2020), pp.~3217--3224.

\bibitem{jiang2022numerical}
{\sc K.~Jiang, S.~Li, and P.~Zhang}, {\em Numerical analysis of computing quasiperiodic systems}, arXiv:2210.04384.

\bibitem{jiang2024numerical}
{\sc K.~Jiang, S.~Li, and P.~Zhang}, {\em Numerical methods and analysis of computing quasiperiodic systems}, SIAM Journal on Numerical Analysis, 62 (2024), pp.~353--375.

\bibitem{jiang2014numerical}
{\sc K.~Jiang and P.~Zhang}, {\em Numerical methods for quasicrystals}, Journal of Computational Physics, 256 (2014), pp.~428--440.

\bibitem{jiang2018numerical}
{\sc K.~Jiang and P.~Zhang}, {\em Numerical mathematics of quasicrystals}, in Proceedings of the International Congress of Mathematicians: Rio de Janeiro 2018, World Scientific, 2018, pp.~3591--3609.

\bibitem{jiang2023accurately}
{\sc K.~Jiang, Q.~Zhou, and P.~Zhang}, {\em Accurately recover global quasiperiodic systems by finite points}, arXiv preprint arXiv:2309.13236,  (2023).

\bibitem{kartashov2021multifrequency}
{\sc Y.~Kartashov, F.~Ye, V.~Konotop, and L.~Torner}, {\em Multifrequency solitons in commensurate-incommensurate photonic moir{\'e} lattices}, Physical Review Letters, 127 (2021), p.~163902.

\bibitem{kelley1995iterative}
{\sc C.~Kelley}, {\em Iterative Methods for Linear and Nonlinear Equations}, SIAM, 1995.

\bibitem{lehoucq2001implicitly}
{\sc R.~Lehoucq}, {\em Implicitly restarted arnoldi methods and subspace iteration}, SIAM Journal on Matrix Analysis and Applications, 23 (2001), pp.~551--562.

\bibitem{lehoucq1998arpack}
{\sc R.~Lehoucq, D.~Sorensen, and C.~Yang}, {\em ARPACK Users' Guide: Solution of Large-scale Eigenvalue Problems with Implicitly Restarted Arnoldi Methods}, SIAM, 1998.

\bibitem{levi2011disorder}
{\sc L.~Levi, M.~Rechtsman, B.~Freedman, T.~Schwartz, O.~Manela, and M.~Segev}, {\em Disorder-enhanced transport in photonic quasicrystals}, Science, 332 (2011), pp.~1541--1544.

\bibitem{levitan1982almost}
{\sc B.~Levitan and V.~Zhikov}, {\em Almost Periodic Functions and Differential Equations}, CUP Archive, 1982.

\bibitem{liao2022adaptive}
{\sc H.~Liao, B.~Ji, and L.~Zhang}, {\em An adaptive bdf2 implicit time-stepping method for the phase field crystal model}, IMA Journal of Numerical Analysis, 42 (2022), pp.~649--679.

\bibitem{liesen2013krylov}
{\sc J.~Liesen and Z.~Strakos}, {\em Krylov Subspace Methods: Principles and Analysis}, Oxford University Press, 2013.

\bibitem{lifshitz1997theoretical}
{\sc R.~Lifshitz and D.~Petrich}, {\em Theoretical model for {F}araday waves with multiple-frequency forcing}, Physical Review Letters, 79 (1997), p.~1261.

\bibitem{lu2019superconductors}
{\sc X.~Lu, P.~Stepanov, W.~Yang, M.~Xie, M.~Aamir, I.~Das, C.~Urgell, K.~Watanabe, T.~Taniguchi, G.~Zhang, A.~Bachtold, A.~H. MacDonald, and D.~K. Efetov}, {\em Superconductors, orbital magnets and correlated states in magic-angle bilayer graphene}, Nature, 574 (2019), pp.~653--657.

\bibitem{meng2023atomic}
{\sc Z.~Meng, L.~Wang, W.~Han, F.~Liu, K.~Wen, C.~Gao, P.~Wang, C.~Chin, and J.~Zhang}, {\em Atomic bose--einstein condensate in twisted-bilayer optical lattices}, Nature, 615 (2023), pp.~231--236.

\bibitem{o2016moire}
{\sc L.~O'Riordan, A.~White, and T.~Busch}, {\em Moir{\'e} superlattice structures in kicked {B}ose-{E}instein condensates}, Physical Review A, 93 (2016), p.~023609.

\bibitem{rodriguez2008computation}
{\sc A.~Rodriguez, A.~McCauley, Y.~Avniel, and S.~Johnson}, {\em Computation and visualization of photonic quasicrystal spectra via {B}loch’s theorem}, Physical Review B, 77 (2008), p.~104201.

\bibitem{saad2003iterative}
{\sc Y.~Saad}, {\em Iterative Methods for Sparse Linear Systems}, SIAM, 2003.

\bibitem{salakhova2021fourier}
{\sc N.~Salakhova, I.~Fradkin, S.~Dyakov, and N.~Gippius}, {\em Fourier modal method for moir{\'e} lattices}, Physical Review B, 104 (2021), p.~085424.

\bibitem{shen2011spectral}
{\sc J.~Shen, T.~Tang, and L.~Wang}, {\em Spectral {M}ethods: {A}lgorithms, {A}nalysis and {A}pplications}, vol.~41, Springer Science \& Business Media, 2011.

\bibitem{simon1982almost}
{\sc B.~Simon}, {\em Almost periodic {S}chr{\"o}dinger operators: a review}, Advances in Applied Mathematics, 3 (1982), pp.~463--490.

\bibitem{stewart2002krylov}
{\sc G.~Stewart}, {\em A {K}rylov-{S}chur algorithm for large eigenproblems}, SIAM Journal on Matrix Analysis and Applications, 23 (2002), pp.~601--614.

\bibitem{van2000krylov}
{\sc H.~A. Van Der~Vorst}, {\em Krylov subspace iteration}, Computing in Science \& Engineering, 2 (2000), pp.~32--37.

\bibitem{wang2020localization}
{\sc P.~Wang, Y.~Zheng, X.~Chen, C.~Huang, Y.~Kartashov, L.~Torner, V.~Konotop, and F.~Ye}, {\em Localization and delocalization of light in photonic {M}oir{\'e} lattices}, Nature, 577 (2020), pp.~42--46.

\bibitem{wang2022convergence}
{\sc T.~Wang, H.~Chen, A.~Zhou, Y.~Zhou, and D.~Massatt}, {\em Convergence of the planewave approximations for quantum incommensurate systems}, arXiv:2204.00994.

\bibitem{watkins2007matrix}
{\sc D.~S. Watkins}, {\em The matrix eigenvalue problem: GR and Krylov subspace methods}, SIAM, 2007.

\bibitem{wright2001large}
{\sc T.~Wright and L.~Trefethen}, {\em Large-scale computation of pseudospectral using arpack and eigs}, SIAM Journal on Scientific Computing, 23 (2001), pp.~591--605.

\bibitem{zhang2021quasi}
{\sc X.~Zhang, Y.~Peng, and D.~Piao}, {\em Quasi-periodic solutions for the general semilinear duffing equations with asymmetric nonlinearity and oscillating potential}, Science China Mathematics, 64 (2021), pp.~931--946.

\bibitem{zhou2019plane}
{\sc Y.~Zhou, H.~Chen, and A.~Zhou}, {\em Plane wave methods for quantum eigenvalue problems of incommensurate systems}, Journal of Computational Physics, 384 (2019), pp.~99--113.

\end{thebibliography}

\end{document}